\documentclass[11pt]{amsart}
\usepackage{times}      
\usepackage{latexsym}   
\usepackage{amssymb}    
\usepackage{amsmath}    
\usepackage{amsthm}     
\usepackage{mathrsfs}   
\newcommand{\beq}{\begin{equation}}
\newcommand{\eeq}{\end{equation}}
\newcommand{\bea}{\begin{eqnarray}}
\newcommand{\eea}{\end{eqnarray}}
\newcommand{\beas}{\begin{eqnarray*}}
\newcommand{\eeas}{\end{eqnarray*}}

\newtheorem{theorem}{Theorem}[section]

\newtheorem{assumption}[theorem]{Assumption}
\newtheorem{proposition}[theorem]{Proposition}

\newtheorem{lemma}[theorem]{Lemma}
\newtheorem{remark}[theorem]{Remark}
\newtheorem{example}[theorem]{Example}
\newtheorem{examples}[theorem]{Examples}
\newtheorem{foo}[theorem]{Remarks}

\def\msh{{\mathscr H}}

\def\me{{\mathbb  E}}

\def\md{{\mathbb D}}
\def\mr{{\mathbb  R}}

\def\mp{{\mathbb  P}}

\numberwithin{equation}{section}

\title[]{Small-time kernel expansion for solutions of stochastic differential equations driven by fractional Brownian motions}
\author{Fabrice Baudoin}
\address{Department of Mathematics\\
       Purdue University \\ West Lafayette, IN 47907}
\email{fbaudoin\@@math.purdue.edu}
\thanks{The first author of this research was supported in part by
NSF Grant DMS 0907326}
\author{Cheng Ouyang}
\address{Department of Mathematics\\
       Purdue University \\ West Lafayette, IN 47907}
\email{couyang\@@math.purdue.edu}

\begin{document}
\maketitle

\begin{abstract}

The goal of this paper is to show that under some assumptions, for a $d$-dimensional fractional Brownian motion with Hurst parameter $H>1/2$, the density of the solution of the stochastic differential equation
\[
X^{x}_t =x + \sum_{i=1}^d \int_0^t V_i
(X^{x}_s) dB^i_s,
\] 
 admits the following asymptotics in small times
\[
p(t;x,y)=\frac{1}{(t^H)^{d}}e^{-\frac{d^2(x,y)}{2t^{2H}}}\bigg(\sum_{i=0}^N c_i(x,y)t^{2iH}+O(t^{ 2(N+1)H})\bigg).
\]
\end{abstract}

\tableofcontents

\section{Introduction}

In this paper, we are interested in the study in small times of
stochastic differential equations on $\mathbb{R}^d$
\begin{equation}
\label{SDEyoung_intro} X^{x}_t =x + \sum_{i=1}^d \int_0^t V_i
(X^{x}_s) dB^i_s
\end{equation}
where $V_i$'s are $C^{\infty}$-bounded vector fields on
$\mathbb{R}^d$ and $B$ is a $d$-dimensional fractional Brownian
motion with Hurst parameter $H >1/2$. 
Since $H>1/2$, the integrals
$\int_0^t V_i (X^{x}_s) dB^i_s$
are understood in the sense of Young's integration (see \cite{Yo}
and \cite{Za}), and it is known (see by e.g. \cite{rascanu}) that an equation like (\ref{SDEyoung_intro}) has one and only one solution. Moreover if for every $x \in \mathbb{R}^d$, the vectors $V_1(x),\cdots,V_d(x)$ form a basis of $\mathbb{R}^d$, then this solution has for every $t>0$, a smooth density with respect to the Lebesgue measure (see \cite{baudoin-hairer} and \cite{saus}).

\

Our main result is the following:

\begin{theorem}\label{mainT}
Let us assume that:
\begin{itemize}
\item For every $x \in \mathbb{R}^d$, the vectors $V_1(x),\cdots,V_d(x)$ form a basis of $\mathbb{R}^d$.
\item There exist smooth and bounded functions $\omega_{ij}^l$ such that:
\[
[V_i,V_j]=\sum_{l=1}^d \omega_{ij}^l V_l,
\]
and
\[
\omega_{ij}^l =-\omega_{il}^j.
\]
\end{itemize}
Then,  in a neighborhood $V$ of $x$, the density function $p(t;x,y)$ of $X^x_t$ in (\ref{SDEyoung_intro}) has the following asymptotic expansion near $t=0$
\begin{align*}
p(t;x,y)=\frac{1}{(t^H)^{d}}e^{-\frac{d^2(x,y)}{2t^{2H}}}\bigg(\sum_{i=0}^N c_i(x,y)t^{2iH}+r_{N+1}(t,x,y)t^{2(N+1)H}\bigg),\quad\quad y\in V.
\end{align*}
Here $d(x,y)$ is the Riemannian distance between $x$ and $y$ determined by the vector fields $V_1,...,V_d$. Moreover, we can chose $V$ such that $c_i(x,y)$ are $C^\infty$ in  $ V\times V\subset \mr^d\times \mr^d$,   and for all multi-indices $\alpha$ and $\beta$ 
$$\sup_{t\leq t_0}\sup_{(x,y)\in V\times V}|\partial^\alpha_x\partial^\alpha_y\partial^k_t r_{N+1}(t,x,y)|<\infty$$
for some $t_0>0$.
\end{theorem}

For $H=1/2$, which corresponds to the Brownian motion case, the above theorem admits numerous proofs. The first proofs were analytic and based on the parametrix method. Such methods do not apply in the present framework since the Markov property for $X_t^x$ is lost whenever $H>1/2$. However,   in the seminal works \cite{Az1} and \cite{azencott}, Azencott introduced probabilistic methods to prove the result. These methods introduced by Azencott were then further developed by Ben Arous and L\'eandre in \cite{Ben Arous}, \cite{benarous2}, \cite{Ben2} and \cite{Leandre2}, in order to cover the  case of hypoelliptic heat kernels.  Let us sketch the strategy of  \cite{benarous2} which is based on the Laplace method on the Wiener space and which is the one adopted in the present paper.

\

The first idea is to consider the scaled stochastic differential equation 
$$dX^\varepsilon_t=\varepsilon\sum_{i=1}^{n}V_i(X^\varepsilon_t)dB^i_t,\quad\quad\mathrm{with}\ X^\varepsilon_0=x_0.$$

We observe that there exist neighborhoods $U$ and $V$ of $x_0$ and a bounded smooth function $F(x,y,z)$ on $U\times V\times \mathbb{R}^n$ such that:

(1) For any $(x,y)\in U\times V$ the infimum 
$$\inf\left\{F(x,y,z)+\frac{d(x,z)^2}{2}, z\in \mathbb{R}^n \right\}=0$$
is attained at the unique point $y$. 

(2) For each $(x,y)\in U\times V$, there exists a ball centered at $y$ with radius $r$ independent of $x,y$ such that $F(x,y,\cdot)$ is a constant outside of the ball.

So, denoting by $p_\varepsilon(x_0,y)$ the density of $X^\varepsilon_1$, by the Fourier inversion formula we have
\begin{align*}
p_\varepsilon(x_0,y)e^{-\frac{F(x_0,y,y)}{\varepsilon^2}}&=\frac{1}{(2\pi)^d}\int e^{-i\zeta\cdot y}d\zeta\int e^{i\zeta\cdot z}e^{-\frac{F(x_0,y,z)}{\varepsilon^2}}p_\varepsilon (x_0,z)dz\\
&=\frac{1}{(2\pi\varepsilon)^d}\int d\zeta \mathbb{E}\left(e^{\frac{i\zeta\cdot(X_1^\varepsilon-y)}{\varepsilon}}e^{\frac{F(x_0,y,X^\varepsilon_1)}{\varepsilon^2}}\right).
\end{align*}
Thus, the asymptotics of $p_t(x_0,y)$ may be understood from the asymptotics when $\varepsilon \to 0$ of 
\[
J_\varepsilon (x_0,y)=\mathbb{E}\left(e^{\frac{i\zeta\cdot(X_1^\varepsilon-y)}{\varepsilon}}e^{\frac{F(x_0,y,X^\varepsilon_1)}{\varepsilon^2}}\right).
\]
Then, by using  the Laplace method on the Wiener space based on the large deviation principle, we get an expansion in powers of $\varepsilon$ of $J_\varepsilon (x_0,y)$ which leads to the expected asymptotics for the density function.

In this work, we follows Ben Arous' approach and show how it may be extended to encompass the fractional Brownian motion case.

\

The rest of this paper is organized as follows. In a preliminary section we remind some known facts about fractional Brownian motion and equations driven by it. In the second section we show how the Laplace method may be carried out in the fractional Brownian motion case and finally in the third section which is the heart of the present paper, we prove Theorem \ref{mainT}. We move the proofs of some technical lemmas to the Appendix.

\begin{remark}
Under the framework of this present work, the Laplace method can be obtained in general hypoelliptic case and without imposing the structure equations on vector fields in Theorem \ref{mainT}.  These two assumptions are imposed to obtain the correct Riemannian distance in the kernel expansion.
\end{remark}
\begin{remark}
When $H>1/2$, to obtain a short-time asymptotic formula for the density of solution to equation (\ref{SDEyoung_intro}) but with drift, one need to work on a version of Laplace method with fractional powers of $\varepsilon$, which will be very heavy and tedious in computation.
\end{remark}
\begin{remark}
When the present work was almost completed, we noticed that a proof for the Laplace method for stochastic differential equation driven by fractional Brownian motion with Hurst parameter $1/3<H<1/2$ became available by Y. Inahama\cite{Inahama2} on mathematics Arxiv. 
\end{remark}

\section{Preliminaries}

\subsection{Stochastic differential equations driven by fractional Brownian motions}

We consider the Wiener space of continuous paths:
\[
\mathbb{W}^{\otimes d}=\left( \mathcal{C} ( [0,T] , \mathbb{R}^d
), (\mathcal{B}_t)_{0 \leq t \leq T}, \mathbb{P} \right)
\]
where:
\begin{enumerate}
\item  $\mathcal{C} ( [0,T] , \mathbb{R}^d
)$ is the space of continuous functions $ [0,T] \rightarrow
\mathbb{R}^d$;
\item  $\left( \beta_{t}\right) _{t\geq 0}$ is the coordinate process defined by
$\beta_{t}(f)=f\left( t\right) $, $f \in \mathcal{C} ( [0,T] ,
\mathbb{R}^d )$;
\item  $\mathbb{P}$ is the Wiener measure;
\item  $(\mathcal{B}_t)_{0 \leq t \leq T}$ is the ($\mathbb{P}$-completed) natural filtration
of $\left( \beta_{t}\right)_{0 \leq t \leq T}$.
\end{enumerate}
A $d$-dimensional fractional Brownian motion with Hurst parameter
$H \in (0,1)$ is a Gaussian process
\[
B_t = (B_t^1,\ldots,B_t^d), \text{ } t \geq 0,
\]
where $B^1,\ldots,B^d$ are $d$ independent centred Gaussian
processes with covariance function
\[
R\left( t,s\right) =\frac{1}{2}\left(
s^{2H}+t^{2H}-|t-s|^{2H}\right).
\]
It can be shown that such a process admits a continuous version
whose paths are H\"older $p$ continuous, $p<H$.
Throughout this paper, we will always consider the `regular' case,  $H > 1/2$.
In this case the fractional Brownian motion can be constructed on the Wiener space by a
Volterra type representation (see \cite{du}). Namely, under the Wiener measure, the process
\begin{equation}\label{Volterra:representation}
B_t = \int_0^t {K_H}(t,s) d\beta_s, t \geq 0
\end{equation}
is a fractional Brownian motion with Hurst parameter $H$, where
\begin{equation*}
{K_H} (t, s)=c_H s^{\frac{1}{2} - H} \int_s^t
(u-s)^{H-\frac{3}{2}} u^{H-\frac{1}{2}} du \;,\qquad t>s.
\end{equation*}
and $c_H$ is a suitable constant.

Denote by $\mathcal{E}$ the set of step functions on $[0,T]$. Let $\mathcal{H}$ be the Hilbert space defined as the closure of $\mathcal{E}$ with respect to the scalar product
$$\langle\mathbf{1}_{[0,t]},\mathbf{1}_{[0,s]}\rangle_{\mathcal{H}}=R_H(t,s).$$
The isometry $K_H^*$ from $\mathcal{H}$ to $L^2([0,T])$ is given by
$$(K_H^*\varphi)(s)=\int_s^T\varphi(t)\frac{\partial K_H}{\partial t}(t,s)dt.$$
Moreover, for any $\varphi\in L^2([0,T])$ we have
$$\int_0^T\varphi(s)dB_s=\int_0^T(K_H^*\varphi)(s)d\beta_s.$$  

We consider the following stochastic differential equation
\begin{equation}
\label{SDEmalliavin} X^{x}_t =x +\int_0^t V_0 (X^x_s)ds+
\sum_{i=1}^d \int_0^t V_i (X^{x}_s) dB^i_s
\end{equation}
where the $V_i$'s are $C^\infty$ vector fields on
$\mathbb{R}^d$  with bounded derivatives to any order and $B$ is the $d$-dimensional fractional Brownian
motion defined by (\ref{Volterra:representation}). Existence and uniqueness of solutions for such equations have widely been  studied and are known to hold in this framework.

\subsubsection{Pathwise estimates}

Let $1/2<\lambda<H$ and denote by $C^\lambda(0, T; \mr^d)$ the space of $\lambda$-H\"{o}lder continuous functions equipped with the norm
$$\|f\|_{\lambda,T}:=\|f\|_\infty+\sup_{0\leq s<t\leq T}\frac{|f(t)-f(s)|}{(t-s)^\lambda},$$
where $\|f\|_\infty:=\sup_{t\in[0,T]}|f(t)|$.

The following remarks will be useful later.
\begin{remark}\label{remark 1}
\ \\
1. It is clear that if $f_1,\ f_2\in C^\lambda$, then $f_1f_2\in C^\lambda$ with $\|f_1f_2\|_{\lambda, t}\leq\|f_1\|_{\lambda, t}\|f_2\|_{\lambda,t}$. Therefore, polynomials of elements in $C^\lambda$ are still in $C^\lambda$. It is also clear that whenever $\varphi$ is a Lipschitz function and $f\in C^\lambda$, we have $\varphi(f)\in C^\lambda.$  

\noindent 2. Let $f\in C^\lambda(0,T;\mr^d)$ and $g: [0, T]\rightarrow \mathcal{M}_{n\times  d}$ be a matrix-valued function and suppose $g\in C^\lambda$. By standard argument (see Terry Lyons\cite{Lyons} for instance),
$$\int_0^. g_s\ df_s\in C^\lambda(0, T; \mr^n)$$
with
$$\left\|\int_0^.g_s\ df_s\right\|_{\lambda, T}\leq C\|g\|_{\lambda, T}\|f\|_{\lambda,T}.$$
In the above $C$ is a constant only depending on $\lambda$ and $T$.
\end{remark}

\begin{lemma}(Hu-Nualart, \cite{Hu})\label{lem:apriori}
Consider the stochastic differential equation (\ref{mainT}), and assume that $\me(|X_0|^p)<\infty$ for all $p\geq 2$. If the derivatives of $V_i$'s are bounded and H\"{o}lder continuous of order $\lambda> 1/H-1$, then
$$\me\left(\sup_{0\leq t\leq T}|X_t|^p\right)<\infty$$
for all $p\geq 2.$ If furthermore $V_i$'s are bounded and $\me(\exp(\lambda|X_0|^q))<\infty$ for any $\lambda>0$ and $q<2H$, then
$$\me\left(\exp \lambda\left(\sup_{0\leq t\leq T}|X_t|^q\right)\right)<\infty$$
for any $\lambda >0$ and $q<2H$.
\end{lemma}

\subsection{Cameron-Martin theorem for fBm}

Consider the classical Cameron-Martin space $\msh=\{h\in P_o(\mr^d): \|h\|_\msh< \infty\}$, where 
$$\|h\|_\msh=\left({\int_0^T |\dot{h}_s|^2ds}\right)^{\frac{1}{2}}.$$ The Cameron-Martin space for the fractional Brownian motion $B$ is
$$\msh_H={K_H}(\msh),$$
where the map $K_H: \msh\rightarrow \msh_H$ is given by
$$(K_Hh)_t=\int_0^tK_H(t, s)\dot{h}_sds,\quad\quad \mathrm{for\ all}\ h\in\msh.$$
The inner product on $\msh_H$ is defined by
$$\langle k_1, k_2\rangle_{\msh_H}=\langle h_1, h_2\rangle_\msh,\quad\quad k_i=K_Hh_i, i=1,2.$$
Hence $K_H$ is an isometry between $\msh$ and $\msh_H$.
\begin{remark} It can be shown that when $\gamma\in\msh_H$, $\gamma$ is $H$-H\"{o}lder continuous. 
\end{remark}

The following Cameron-Martin theorem is known (see \cite{du}).
\begin{theorem}[Cameron-Martin theorem for fBm]\label{C-M}
Let $B^k=B+k$ be the shifted fractional Brownian motion, where $k\in \msh_H$ is a Cameron-Martin path. The law $\mp_H^k$ of $B^k$ and the law $\mp_H$ of $B$ are mutually absolutely continuous. Furthermore, the Radon-Nikodym derivative is given by
\begin{align*}
\frac{d\mp_H^k}{d\mp_H}=\exp\left[-\int_0^T(K_H^*)^{-1}(\dot{h})_sdB_s-\frac{1}{2}\|k\|^2_{\msh_H}\right],
\end{align*}
In the above, $h=(K_H)^{-1}k$ and the integral against $B$ is understood as Young's integral.\end{theorem}
\subsection{Large deviation principle for fBm}  The following large deviation principle for stochastic differential equation driven by fractional Brownian motion is a special case of Proposition 19.14 in Friz-Victoir\cite{Friz} (see also \cite{millet}).

\begin{proposition}
Fix $\lambda\in(1/2,H)$.
Let $X^\varepsilon$ be the solution to the following stochastic differential equations driven by fBm $B$
\begin{align}\label{sde fBm}X_t^\varepsilon=x_0+\int_0^t V_0(X_s)ds+\sum_{i=1}^{d}\varepsilon\int_0^tV_i(X_s)dB_s^i\end{align}
where $V_i$'s are $C^\infty$ vector fields on $\mr^d$ with bounded derivatives to any order. 
The process $X^\varepsilon$ satisfies a large deviation principle, in $\lambda$-H\"{o}lder topology, with good rate function given by
$$\Lambda(\phi)=\inf\{\bar{\Lambda}(\gamma): \phi=I(\gamma) \}$$
where $I$ is the It\^{o} map given by (\ref{sde fBm}) with $\varepsilon$ being replaced by $1$, and $\bar\Lambda$ is given by
$$\bar{\Lambda}(\gamma)=\left\{\begin{array}{lll}
  \frac{1}{2}\|\gamma\|_{\msh_H}^2&\mathrm{if}\ \gamma\in\msh_H, \\
  \\
  +\infty&\mathrm{otherwise.}
\end{array}\right.$$
\end{proposition}

\section{Laplace method}
Consider the following stochastic differential equation driven by fractional Brownian motion on $\mr^d$:
\begin{align*}X_t^\varepsilon=x_0+\int_0^t V_0(X_s)ds+\sum_{i=1}^{d}\varepsilon\int_0^tV_i(X_s)dB_s^i.\end{align*}
For the convenience of our discussion, in what follows, we write the above equation in the following form
\begin{align*}
X_t^\varepsilon=x+\varepsilon\int_0^t\sigma(X^\varepsilon_s)dB_s+\int_0^tb(\varepsilon,X_s^\varepsilon)ds,
\end{align*}
where $\sigma$ is a smooth $d\times d$ matrix and $b$ a smooth function from $\mr^+\times\mr^d$ to $\mr^d$. We also assume that $\sigma$ and $b$ have bounded derivatives to any order.

Let $F$ and $f$ be two smooth functionals with smooth derivatives to any order. We are interested in studying the asymptotic behavior of 
$$J(\varepsilon)=\me\big[f(X^\varepsilon)\exp\{-F(X^\varepsilon)/\varepsilon^2\}\big]$$
as $\varepsilon \downarrow 0$. Indeed, the following theorem is the main result of this section.

\begin{theorem}
Under the assumption H 1 and H 2 below, we have
$$J(\varepsilon)=e^{-\frac{a}{\varepsilon^2}}e^{-\frac{c}{\varepsilon}}\bigg(\alpha_0+\alpha_1\varepsilon+...+\alpha_N\varepsilon^N+O(\varepsilon^{N+1})\bigg).$$
Here
$$a=\inf\{F+\Lambda(\phi), \phi\in P(\mr^d)\}=\inf\{F\circ\Phi(k)+1/2|k|^2_{\msh_H}, k\in\msh_H\}$$
and
$$c=\inf\big\{ dF(\phi_i)Y_i, i\in\{1,2,...,n\}\big\},$$
where $Y_i$ is the solution of
$$dY_i(s)=\partial_x\sigma(\phi_i(s))Y_i(s)d\gamma_i(s)+\partial_\varepsilon b(0,\phi_i(s))ds+\partial_x b(0,\phi_i(s))Y_i(s)ds$$
with $Y_i(0)=0$.
\end{theorem}
For each $k\in \msh_H$, denote by $\Phi(k)$ the solution to the following deterministic differential equation
\begin{align}\label{Phi}du_t=\sigma(u_t)dk_t+b(0,u_t)dt,\quad \mathrm{with}\ u_0=x.\end{align}
\begin{lemma}\label{th: EU}
Let $\Phi$ be defined as above, we have
$$\Lambda(\phi)=\inf\left\{\frac{1}{2}\|k\|_{\msh_H}^2, \phi=\Phi(k), k\in\msh_H\right\}.$$
Moreover, if $\Lambda(\phi)<\infty$, there exists a unique $k\in\msh_H$ such that $\Phi(k)=\phi$ and $\Lambda(\phi)=1/2\|k\|_{\msh_H}^2.$
\end{lemma}
\begin{proof}
The first statement is apparent. For the second statement, we only need to notice that if $$\phi=\Phi(k_1)=\Phi(k_2),\quad\quad k_1, k_1\in\msh_H,$$
then
$$\int_0^t\sigma(\phi_s)d(k_1-k_2)_s=0,\quad\quad \ t\in[0,T],$$
which implies that $k_1=k_2$, since we assume that columbs of $\sigma$ are linearly independent. The proof is therefore completed.
\end{proof}


Throughout our discussion we make the following assumptions:
\begin{assumption}\label{assumption Laplace}
\ \\
\begin{itemize}
\item H 1: $F+\Lambda$ attains its minimum at finite number of paths $\phi_1, \phi_2, ..., \phi_n$ on $P(\mr^d)$.
\ \\
\item H 2: For each $ i\in\{1,2,...,n\}$, we have $\phi_i=\Phi(\gamma_i)$ and $\gamma_i$ is a non-degenerate minimum of the functional $F\circ\Phi+1/2\|\cdot\|^2_{\msh_H}$, i.e.:
$$ \forall k\in\msh_H-\{0\},\quad d^2(F\circ\Phi+1/2\| \cdot\|^2_{\msh_H})(\gamma_i)k^2>0.$$
\end{itemize}
\end{assumption}

\begin{lemma} Under assumption H 1, we have $$a\stackrel{\mathrm{def}}{=}\inf\{F+\Lambda(\phi), \phi\in P(\mr^d)\}=\inf\left\{F\circ\Phi(k)+\frac{1}{2}\|k\|^2_{\msh_H}, k\in\msh_H\right\},$$ and the minimum is attained at $n$ paths $\gamma_1, \gamma_2,...,\gamma_n\in\msh_H$ such that $$\Phi(\gamma_i)=\phi_i$$ and $$\frac{1}{2}\|\gamma_i\|^2_{\msh_H}=\Lambda(\Phi(\gamma_i)).$$
\end{lemma}
\begin{proof}
This is a direct corollary of Lemma \ref{th: EU}.
\end{proof}

Assumption H 2 has a simple interpretation as follows. Let $\gamma$ be one of the $\gamma_i$'s above. Define a bounded self-adjoint operator on $\msh$ by
$$d^2 F\circ\Phi(\gamma)(K_Hh^1,K_Hh^2)=(Ah^1,h^2)_\msh,\quad\quad \mathrm{for}\ h^1,\ h^2\in\msh.$$
\begin{lemma}\label{H2}
The bounded self-adjoint operator $A$ is Hilbert-Schmidt. 
\end{lemma}
\begin{proof}
The proof is similar to that in Ben Arous\cite{Ben Arous} but with slight modification. Thus we only sketch the proof here . In what follows, $k$ always denotes an element in $\msh_H$ and $h=K_H^{-1}k$ its corresponding element in $\msh$.

For any $k^1, k^2\in\msh_H$, we have
\begin{align*}d^2F\circ\Phi(\gamma)(K_Hh^1,K_Hh^2)&=d^2F\circ\Phi(\gamma)(k^1,k^2)\\
&=d^2F(d\Phi(\gamma)k^1, d\Phi(\gamma)k^2)+dF(\phi)(d^2\Phi(\gamma)(k^1,k^2)).
\end{align*}

Let
$$\phi=\Phi(\gamma)\quad\quad\ \mathrm{and}\quad \ \chi(k)=d\Phi(\gamma)k.$$
It can be shown (cf. Ben Arous\cite{Ben Arous}),
$$d\phi_t=\sigma(\phi_t)d\gamma_t+b(0,\phi_t)dt,\quad\quad\mathrm{with}\ \phi_0=x,$$
$$d\chi_t=\sigma(\phi_t)dk_t+\partial_x\sigma(\phi_t)\chi_td\gamma_t+\partial_xb(0,\phi_t)\chi_tdt,\quad\quad\mathrm{with}\ \chi_0=0,$$
and
\begin{align*}
d^2\Phi(\gamma)(k^1,k^2)(t)&=\int_0^1Q(t,s)\partial_x\sigma(\phi_s)\big(\chi(k^1)_sdk^2_s+\chi(k^2)_sdk^1_s\big)\\
&\ \ +\int_0^t\partial^2_{xx}\sigma(\phi_s)\big(\chi(k^1)_s,\chi(k^2)_s\big)d\gamma_s+\int_0^t\partial^2_{xx}b(0,\phi_s)\big(\chi(k^1)_s,\chi(k^2)_s\big)ds.
\end{align*}
Here $Q(t,s)$ takes the form
$$Q(t,s)=\partial_x\phi_t(x)\partial_x\phi_s(x)^{-1}.$$
Moreover, we have
\begin{align}\label{chi2}
\chi_t(k)&=\int_0^tQ(t,s)\sigma(\phi_s)dk_s\\
&=\int_0^t\left(\int_u^t Q(t,s)\sigma(\phi_s)\frac{\partial K_H(s,u)}{\partial s} ds\right)\dot{h}_udu\nonumber
\end{align}
Set
\begin{align}\label{V}V(h^1,h^2)(t)&=\int_0^tQ(t,s)\partial_x\sigma(\phi_s)\big(\chi(K_Hh^1)_sd(K_Hh^2)_s+\chi(K_Hh^2)_sd(K_Hh^1)_s\big)\\
&=\int_0^tQ(t,s)\partial_x\sigma(\phi_s)\big(\chi(k^1)_sdk^2_s+\chi(k^2)_sdk^1_s\big)\nonumber\\
&=\int_0^t\int_u^tQ(t,s)\partial_x\sigma(\phi_s)\frac{\partial K_H(s,u)}{\partial s}\big(\chi(k^1)_sh^2_u+\chi(k^2)_sh^1_u\big)dsdu.\nonumber
\end{align}
Define a bounded self-adjoint operator $\tilde{A}$ from $\msh$ to $\msh$ by
$$ dF(\phi)(V(h^1,h^2))=(\tilde{A} h^1, h^2)_\msh$$ 
We conclude that $\tilde{A}$ is Hilbert-Schmidt since, by (\ref{chi2}) and (\ref{V}), it is defined from a $L^2$ kernel. Therefore, to complete the proof, it suffices to show that $A-\tilde{A}$ is Hilber-Schmidt. By the same argument as in Ben Arous\cite{Ben Arous}, we only need to show
$$\|d\Phi(\gamma)K_Hh\|_\infty=\|\chi(K_Hh)\|_\infty\leq C\|h\|_\infty,\quad\quad\mathrm{for\ all}\ h\in\msh.$$
Indeed, by an easy application of Gronwall inequality to the equation for $\chi$, we have
$$\|d\Phi(\gamma)(K_Hh)\|_\infty\leq \|K_Hh\|_\infty.$$
Moreover, since
$$(K_Hh)_t=\int_0^t K_H(t,s)\dot{h}_sds,$$
and note $\partial K_H(t,s)/\partial s\in L^1$,we have
\begin{align*}
|K_Hh|_t\leq\left|\int_0^tK_H(t,s)\dot{h}_sds\right|
=\left|\int_0^t h_s \frac{\partial K_H(t,s)}{\partial s}ds\right|
\leq \|h\|_\infty \int_0^t  \left|\frac{\partial K_H(t,s)}{\partial s}\right|ds,
\end{align*}
The proof is completed.

\end{proof}
From the above lemma, assumption H 2 simply means that the smallest eigenvalue of $A$ is attained and is strictly greater that $-1$.

\subsection{Localization around the minimum}
By the large deviation principle, the sample paths that has contribution to the asymptotics of $J(\varepsilon)$ lie in the neighborhoods of the minimizers  of $F+\Lambda$. More precisely,
\begin{lemma} For $\rho>0$, denote by $B(\phi_i, \rho)$ the open ball (under $\lambda$-H\"{o}lder topology) centered at $\phi_i$ with radius $\rho$. There exist $d>a$ and $\varepsilon_0>0$ such that for all $\varepsilon\leq \varepsilon_0$
$$\left|J(\varepsilon)-\me\left[f(X_T^\varepsilon)e^{-F(X_T^\varepsilon)/\varepsilon^2}, X^\varepsilon\in\bigcup_{1\leq i\leq n}B(\phi_i,\rho)\right]\right|\leq e^{-d/\varepsilon^2}.$$
\end{lemma}
\begin{proof}
This is a consequence of the large deviation principle.
\end{proof}
Assume that $n=1$, i.e., $F+\Lambda$ attains its minimum at only one path $\phi$. Let $$J_\rho(\varepsilon)=\me\left[f(X_T^\varepsilon)e^{-F(X_T^\varepsilon)/\varepsilon^2}, X^\varepsilon\in B(\phi,\rho)\right].$$
The above lemma tells us that to study the asymptotic behavior of $J(\varepsilon)$ as $\varepsilon\downarrow 0$, it is suffice to study that of $J_\rho(\varepsilon)$.

\subsection{Stochastic Taylor expansion and Laplace approximation}
In this section, we prove an asymptotic expansion for $J_\rho(\varepsilon)$. 

Let $\phi$ be the unique path that minimizes $F+\Lambda$. There exists a $\gamma\in\msh_H$ such that $$\phi=\Phi(\gamma),\quad\quad \mathrm{and}\ \Lambda(\phi)=\frac{1}{2}\|\gamma\|^2_{\msh_H},$$ and for all $k\in\msh_H-\{0\}$:
$$d^2(F\circ\Phi+\frac{1}{2}\|\ \|^2_{\msh_H})(\gamma)k^2>0.$$

Let
$$\chi(k)=d\Phi(\gamma)k\quad\quad\mathrm{and}\ \ \psi(k,k)=d^2\Phi(\gamma)(k,k).$$
We have
\begin{align}\label{chi}
d\chi_t=\sigma(\phi_t)dk_t+\partial_x\sigma(\phi_t)\chi_t d\gamma_t+\partial_xb(0,\phi_t)\chi_tdt,
\end{align}
and
\begin{align}\label{psi}
d\psi_t=&2\partial_x\sigma(\phi_t)\chi_tdk_t+\partial^2_{xx}\sigma(\phi_t)\chi^2_td\gamma_t+\partial_x\sigma(\phi_t)\psi_td\gamma_t\\
&+\partial^2_{xx}b(0,\phi_t)\chi^2_tdt+\partial_xb(0,\phi_t)\psi_tdt.\nonumber
\end{align}
Here $\chi_0=\phi_0=0$. These formula will be useful later.

Consider the following stochastic differential equation
$$Z^\varepsilon_t=x+\int_0^t\sigma(Z^\varepsilon_s)(\varepsilon dB_s+d\gamma_s)+\int_0^tb(\varepsilon,Z^\varepsilon_s)ds.$$
It is clear that $Z^0=\phi$. Denote $Z_t^{m,\varepsilon}=\partial_\varepsilon^mZ^\varepsilon_t$ and consider the Taylor expansion with respect to $\varepsilon$ near $\varepsilon=0$, we obtain
$$Z^\varepsilon=\phi+\sum_{j=0}^N\frac{g_j\varepsilon^j}{j!}+\varepsilon^{N+1}R^\varepsilon_{N+1},$$
where $g_j=Z^{j,0}.$ Explicitly, we have
\begin{align*}
dg_1(s)=\sigma(\phi_s)dB_s+\partial_x\sigma(\phi_s)g_1(s)d\gamma_s+\partial_xb(0,\phi_s)g_1(s)ds+\partial_\varepsilon b(0,\phi_s)ds.
\end{align*}

Similar to the Brownian motion case, we have the following estimates, the proof of which is postponed to Appendix. 
\begin{lemma}\label{th: key lemma1}For any $t\in[0,T]$, there exists a constant $C>0$ such that for $r$ large enough we have
\begin{align*}
&\mp\{\|g_1\|_{\lambda,t}\geq r\}\leq\exp\left\{-\frac{C r^2}{t^{2H}}\right\}\\
&\mp\{\|g_2\|_{\lambda,t}\geq r\}\leq\exp\left\{-\frac{C r}{t^{2H}}\right\}.
\end{align*}
and
\begin{align*}
&\mp\{\|\varepsilon R_1^\varepsilon\|_{\lambda, t}\geq r; t\leq T^\varepsilon\}\leq \rho\\
&\mp\{\|\varepsilon R_2^\varepsilon\|_{\lambda,t}\geq r; t\leq T^\varepsilon\}\leq \exp\left\{-\frac{C r^2}{\rho t^{2H}}\right\}\nonumber \\
&\mp\{\|\varepsilon R_3^\varepsilon\|_{\lambda,t}\geq r; t\leq T^\varepsilon\}\leq \exp\left\{-\frac{C r}{\rho t^{2H}}\right\},\nonumber
\end{align*}
Here $T^\varepsilon$ is the first exist time of $Z^\varepsilon$ from  $B(\phi,\rho)$.
\end{lemma}

Let $\theta(\varepsilon)=F(Z_T^\varepsilon)$. By Taylor expansion of $\theta(\varepsilon)$ with respect to $\varepsilon$, we obtain
$$\theta(\varepsilon)=\theta(0)+\varepsilon\theta'(0)+\varepsilon^2U(\varepsilon).$$
Here
$$U(\varepsilon)=\int_0^1(1-v)\theta''(\varepsilon v)dv,\quad\mathrm{and}\ \ \ \theta(0)=F(\phi).$$

\begin{lemma}\label{th: theta 1}With the above notation, we have
\begin{align*}\theta'(0)&=dF(\phi)g_1
=-\int_0^T\big((K^*_H)^{-1}\dot{(K_H^{-1}\gamma)}\big)_sdB_s+dF(\phi)Y.\end{align*}
Here $Y$ is the solution of
$$dY_s=\partial_x\sigma(\phi_s)Y_sd\gamma_s+\partial_\varepsilon b(0,\phi_s)ds+\partial_x b(0,\phi_s)Y_sds,\quad\ Y(0)=0.$$
\end{lemma} 
\begin{proof}
By an easy application of the Gronwall's inequality to (\ref{chi}), we have for any $k\in\msh_H$, 
\begin{align}\label{continuity}\|d\Phi(\gamma)k\|_\infty\leq C\|k\|_\infty\end{align} for some positive constant $C$. Therefore, $d\Phi(\gamma)$ can be extended continuously to an operator on $P(\mr^d)$. We have 
\begin{align*}g_1=d\Phi(\gamma)B+Y.\end{align*}
On the other hand, since $\gamma$ is a critical point of $F\circ\Phi+1/2\|\cdot \|^2_{\msh_H}$ and note $\|k\|_{\msh_H}=\|K_H^{-1}k\|_\msh$, we have
\begin{align}\label{idty}dF(\phi)(d\Phi(\gamma)k)&=-\int_0^T\dot{(K_H^{-1}\gamma)}_s\dot{(K_H^{-1}k)}_sds\\
&=-\int_0^T\big((K^*_H)^{-1}\dot{(K_H^{-1}\gamma)}\big)_sdk_s\nonumber
\end{align}
for all $k\in\msh_H.$ The second equation above can be seen as follows. Denote by
$$h=K_H^{-1}k.$$
We have
\begin{align*}
\int_0^T\big((K^*_H)^{-1}\dot{(K_H^{-1}\gamma)}\big)_sdk_s
=&\int_0^T\big((K^*_H)^{-1}\dot{(K_H^{-1}\gamma)}\big)_s\int_0^s\frac{\partial K_H}{\partial s}(s,u)\dot{h}_ududs\\
=&\int_0^T\dot{h}_u\int_u^T\big((K^*_H)^{-1}\dot{(K_H^{-1}\gamma)}\big)_s\frac{\partial K_H}{\partial s}(s,u)ds\\
=&\int_0^T\dot{h}_u\dot{(K_H^{-1}\gamma)}_udu\\
=&\int_0^T\dot{(K_H^{-1}\gamma)}_s\dot{(K_H^{-1}k)}_sds.
\end{align*}

From (\ref{continuity}) and (\ref{idty}) we conclude that the path $(K^*_H)^{-1}\dot{(K_H^{-1}\gamma)}$ has bounded variation and hence, by passing to limit, we obtain
$$dF(\phi)(d\Phi(\gamma)B)=-\int_0^T\big((K^*_H)^{-1}\dot{(K_H^{-1}\gamma)}\big)_sdB_s.$$
The proof is completed.

\end{proof}

Now, by Theorem \ref{C-M} we have
\begin{align*}&J_\rho(\varepsilon)\\
=&\me\left[f(Z^\varepsilon)\exp\left(-\frac{F(Z^\varepsilon)}{\varepsilon^2}\right)\exp\left(-\frac{1}{\varepsilon}\int_0^T\big((K^*_H)^{-1}(\dot{K_H^{-1}{\gamma}})\big)_sdB_s-\frac{\|\gamma\|^2_{\msh_H}}{2\varepsilon^2}\right); Z^\varepsilon\in B(\phi,\rho)\right]\\
=&\me\big[V(\varepsilon);Z^\varepsilon\in B(\phi,\rho)\big]\exp\left[-\frac{1}{\varepsilon^2}\left(F(\phi)+\frac{1}{2}\|\gamma\|^2_{\msh_H}\right)\right]\exp\left[-\frac{dF(\phi)Y}{\varepsilon}\right]\\
=&\me\big[V(\varepsilon); Z^\varepsilon B(\phi,\rho)\big]\exp\left[-\frac{a}{\varepsilon^2}\right]\exp\left[-\frac{dF(\phi)Y}{\varepsilon}\right].
\end{align*}
In the above
$$V(\varepsilon)=f(Z^\varepsilon)e^{-U(\varepsilon)}.$$

To prove the Laplace approximation, it now suffices to estimate $\me\big[V(\varepsilon); Z^\varepsilon\in B(\phi,\rho)\big]$. For this purpose, we need the following two technical lemmas.

\begin{lemma}\label{th: key lemma2}
Let$$\theta(\varepsilon)=F(Z^\varepsilon)=\theta(0)+\varepsilon\theta'(0)+\varepsilon^2U(\varepsilon)$$
where $$U(\varepsilon)=\int_0^1(1-v)\theta''(\varepsilon v)dv,\quad\quad\mathrm{and}\quad\theta(0)=F(\phi).$$ There exist $\beta>0$ and $\varepsilon_0>0$ such that
$$\sup_{0\leq \varepsilon\leq\varepsilon_0}\me\left(e^{-(1+\beta)U(\varepsilon)}; t\leq T^\varepsilon\right)<\infty.$$
\end{lemma}
\begin{proof}
See Appendix.
\end{proof}
\begin{lemma}\label{th: key lemma3}
For all $m>0$ and $p\geq 2$, there exists an $\varepsilon_0>0$ such that
$$\sup_{\varepsilon\leq\varepsilon_0}\me\left(\sup_{t\in[0,1]}|\partial_\varepsilon^m Z_t^\varepsilon|^p\right)<\infty.$$
\end{lemma}
\begin{proof}
This is a consequence of Lemma \ref{lem:apriori}.
\end{proof}

Denote $V^{(m)}(\varepsilon)=\partial^m_\varepsilon V(\varepsilon)$. By Lemma \ref{th: key lemma2} and Lemma \ref{th: key lemma3}, one can show
$$\me|V^{(m)}(0)|^p<\infty,\quad\quad\quad\mathrm{for\ all}\ p>1, m>0.$$
Consider the stochastic Taylor expansion for $V(\varepsilon)$
$$V(\varepsilon)=\sum_{m=0}^N \frac{\varepsilon^mV^{(m)}(0)}{m!}+\varepsilon^{N+1}S^\varepsilon_{N+1}$$
where
$$S^\varepsilon_{N+1}=\int_0^1\frac{V^{(N+1)}(\varepsilon v)(1-v)^N}{N!}dv.$$
It can be shown, again by Lemma \ref{th: key lemma2} and Lemma \ref{th: key lemma3} (cf, Ben Arous\cite{Ben Arous}), 
$$\sup_{0\leq\varepsilon\leq\varepsilon_0}\me\big[|S^\varepsilon_{N+1}|; Z^\varepsilon\in B(\phi),\rho)\big]<\infty.$$
Thus we conclude that
$$\me\big[V(\varepsilon); Z^\varepsilon\in B(\phi,\rho)\big]=\sum_{m=0}^{N}\alpha_m\varepsilon^m+O(\varepsilon^{N+1}).$$
Moreover, one can show
$$\alpha_m=\frac{\me V^{(m)}(0) }{m!}.$$

\section{Short-time expansion for transition density}
We now arrive to the heart of our study and are interested in obtaining a short-time expansion for the density function of $X_t$, where
\begin{align}\label{equ X}dX_t=\sum_{i=1}^{d}V_i(X_t)dB^i_t, \quad X_0=x\end{align}
Here $V_i$'s are $C^\infty$ vector fields on $\mr^d$ with bounded derivatives to any order. Throughout this section, we shall also make the following assumption on the vector fields $V_i$'s.

\begin{assumption}\label{H}

\

\begin{itemize}
\item For every $x \in \mathbb{R}^d$, the vectors $V_1(x),\cdots,V_d(x)$ form a basis of $\mathbb{R}^d$.
\item There exist smooth and bounded functions $\omega_{ij}^l$ such that:
\[
[V_i,V_j]=\sum_{l=1}^d \omega_{ij}^l V_l,
\]
and
\[
\omega_{ij}^l =-\omega_{il}^j.
\]
\end{itemize}
\end{assumption}

The first assumption means that the vector fields  form an elliptic differential system. As a consequence of  Baudoin and Hairer\cite{baudoin-hairer},  it is known that the law of $X_t$, $t>0$, admits therefore a smooth density $p(t;x,y)$ with respect to Lebesgue measure. The second assumption is of geometric nature and actually means that the Levi-Civita connection associated with the Riemannian structure given by the vector fields $V_i$'s is 
\[
\nabla_X Y=\frac{1}{2} [X,Y].
\]
In a Lie group structure, this is equivalent to the fact that the Lie algebra is of compact type. We will see the use of this assumption in a section below.

The following theorem is the main result of our paper.
\begin{theorem}\label{main}
Fix $x\in \mr^d$. Assume that the assumption \ref{H} is satisfied, then in a neighborhood $V$ of $x$, the density function $p(t;x,y)$ of $X_t$ in (\ref{equ X}) has the following asymptotic expansion near $t=0$
\begin{align*}
p(t;x,y)=\frac{1}{(t^H)^{d}}e^{-\frac{d^2(x,y)}{2t^{2H}}}\bigg(\sum_{i=0}^N c_i(x,y)t^{2iH}+r_{N+1}(t,x,y)t^{2nH}\bigg),\quad\quad y\in V.
\end{align*}
Here $d(x,y)$ is the Riemannian distance between $x$ and $y$ determined by $V_1,...,V_d$. Moreover, we can chose $V$ such that $c_i(x,y)$ are $C^\infty$ in  $ V\times V\subset \mr^d\times \mr^d$,   and for all multi-indices $\alpha$ and $\beta$ 
$$\sup_{t\leq t_0}\sup_{(x,y)\in V\times V}|\partial^\alpha_x\partial^\alpha_y\partial^k_t r_{N+1}(t,x,y)|<\infty$$
for some $t_0>0$.
\end{theorem}

Once the Laplace approximation in the previous section is obtained, the proof of the above theorem is actually quite standard and follows closely the argument given, for instance, in Ben Arous\cite{benarous2}. Thus, for most of the lemmas in what follows, we only outline the proofs but stress the main differences with Brownian motion case.

\subsection{Preliminaries in differential geometry} The vector fields $V_1, V_2,..., V_d$ on $\mr^d$ determine a natural Riemannian metric $g=(g_{ij})$ on $\mr^d$ under which $V_1(x), V_2(x),..., V_d(x)$ form an orthonormal frame at each point $x\in\mr^d$. More explicitly, let $\sigma$ be the $d\times d$ matrix formed by
$$ \sigma(x)=(V_1(x), V_2(x),...,V_d(x)).$$
Denote by $\Gamma$ the inverse matrix of $\sigma\sigma^*$. Then the Riemannian metric $g$ is given by
$$g_{ij}=\Gamma_{ij},\quad\quad 1\leq i,j\leq d.$$ 
Throughout our discussion, we denote by $M$ the Riemannian manifold $\mr^d$ equipped with the metric $g$ specified above.
The Riemannian distance between any two points $x,y$ on $M$ is denoted by $d(x,y)$. 
We recall that
\[
d(x,y)=\inf_{\gamma \in \mathcal{C}(x,y)} \int_0^1 \sqrt{g_{\gamma(s)} (\gamma'(s),\gamma'(s)) }ds
\]
where $\gamma \in \mathcal{C}(x,y)$ denotes the set of absolutely continuous curves $\gamma:[0,1]\rightarrow \mathbb{R}^d$, such that $\gamma(0)=x, \gamma(1)=y$.

More analytically, this distance may also be defined as
\[
d(x,y)=\sup \{ f(x)-f(y), f\in C_b^{\infty}(\mathbb{R}^d), \sum_{i=1}^d (V_i f)^2 \le 1 \},
\]
where $C_b^{\infty}(\mathbb{R}^d)$ denotes the set of smooth and bounded functions on $\mathbb{R}^d$. Since the vector fields $V_1,\cdots, V_d$ are Lipschitz it is well-known that this distance is complete and that the Hopf-Rinow theorem holds (that is closed balls are compact). 

\

Due to the second assumption \ref{H}, the geodesics are easily described. If $k: \mathbb{R}_{\ge 0} \rightarrow \mathbb{R}$ is a $\alpha$-H\"older  path with $\alpha >1/2$ such that $k(0)=0$, we denote by $\Phi(x,k)$ the solution of the ordinary differential equation:
\[
x_t=x+\sum_{i=1}^d \int_0^t V_i(x_s) dk^i_s.
\]
Whenever there is no confusion, we always suppress the starting point $x$ and denote it simply by $\Phi(k)$ as before.
\begin{lemma}
$\Phi(x,k)$ is a geodesic if and only if $k(t)=t u$ for some $u \in \mathbb{R}^d$.
\end{lemma}

\begin{proof}
It is well-known that geodesics $c$ are smooth and solutions of the equation
\[
\nabla_{c'} c' =0,
\]
where $\nabla$ is the Levi-Civita connection.
Therefore, in order $\Phi(k)$ to be a geodesic, we first see that $k$ needs to be smooth and then that
\[
\nabla_{\sum_{i=1}^d  V_i(x_s) \dot{k}^i_s}\sum_{i=1}^d  V_i(x_s) \dot{k}^i_s=0.
\]
Now, due to the structure equations
\[
[V_i,V_j]=\sum_{l=1}^d \omega_{ij}^l V_l,
\]
the Christoffel's symbols of the connection are given by
\[
\Gamma_{ij}^l=\frac{1}{2} \left(\omega_{ij}^l +\omega_{li}^j + \omega_{lj}^i
\right)=\frac{1}{2} \omega_{ij}^l.
\]
So the equation of geodesics may be rewritten
\[
\sum_{l=1}^d \frac{d^2 k^l_s}{ds^2} V_l(x_s)+\sum_{i,j,l=1}^d \omega_{ij}^l \dot{k}^i_s \dot{k}^j_s V_l(x_s)=0.
\]
Due to the skew-symmetry $\omega_{ij}^l=-\omega_{ji}^l$ we get
\[
\frac{d^2 k^l_s}{ds^2} =0,
\]
which leads to the expected result.
\end{proof}

As a consequence of the previous  lemma, we then have the following key result:

\begin{proposition}\label{th: ministance}
Let $T>0$. For $x,y \in \mathbb{R}^d$,
\[
\inf_{k \in \mathcal{H}_H, \Phi_T(x,k)=y } \| k \|^2_{\msh_H}=\frac{d^2(x,y)}{T^{2H}} . 
\]
\end{proposition}

\begin{proof}

In a first step we prove 
\[
\frac{d^2(x,y)}{T^{2H}}  \le \inf_{k \in \mathcal{H}_H, \Phi_T(x,k)=y } \| k \|^2_{\msh_H}. 
\]
Let $k \in \msh_H$ such that $\Phi_0(k)=x, \Phi_T(k)=y$. Denote by $z$ the solution of the equation
\[
dz_t=\sum_{i=1}^d V_i(z_t) dk^i_t, \quad 0 \le t \le T.
\]
We have therefore:
\[
z_0=x, \quad z_T=y.
\]
Let now $f  \in C_b^{\infty}(\mathbb{R}^d)$ such that $ \sum_{i=1}^d (V_i f)^2 \le 1$ . By the change of variable formula, we get
\[
f(y)-f(x) =\sum_{i=1}^d \int_0^T V_if (z_t) dk^i_t.
\]
Since $k \in \msh_H$, we can find $h$ in the Cameron-Martin space of the Brownian motion such that
\[
k_t=\int_{0}^t K_H (t,s) \dot{h}_s ds.
\]
Integrating by parts, we have then
\[
\int_0^T V_if (z_t) dk^i_t=\int_0^T \left( \int_s^T \frac{\partial K_H}{\partial t}(t,s) V_if(z_t) dt \right) \dot{h}^i_s ds.
\]
Therefore from Cauchy-Schwarz inequality, the isometry between $\msh$ and $\msh_H$ and the fact that $ \sum_{i=1}^d (V_i f)^2 \le 1$, we deduce that
\[
(f(y)-f(x))^2 \le R(T,T) \| \dot{h} \|^2_{L^2([0,1])}= T^{2H}  \| k \|^2_{\msh_H}.
\]
Thus 
\[
\frac{d^2(x,y)}{T^{2H}} \le \inf_{k \in \msh_H, \Phi_T(x,k)=y } \| k \|^2_{\msh_H}.
\]

\

We now prove the converse inequality.

We first assume that $y$ is close enough to $x$ so that there exist $(y_1,\cdots,y_d) \in \mathbb{R}^d$ that satisfy
\[
y=\exp \left(\sum_{i=1}^d y_i V_i \right) (x).
\]
Let
\[
k^i_t=\frac{\int_0^t K_H(t,s)K_H(T,s)ds}{T^{2H} } y_i=\frac{R(t,T)}{T^{2H} } y_i.
\]
In that case, it is easily seen that
\[
\Phi(k)(t)=\exp \left( \sum_{i=1}^d  \frac{R(t,T)}{T^{2H} } y_i V_i \right) (x).
\]
In particular,
\[
\Phi_0(k)=x, \Phi_1(k)=y.
\]
Moreover,
\[
\| k \|^2_{\msh_H}=\frac{\sum_{i=1}^d y_i^2}{T^{2H}}=\frac{d^2(x,y)}{T^{2H}}.
\]
As a consequence
\[
\inf_{k \in \msh_H, \Phi_T(x,k)=y } \| k \|^2_{\msh_H}\le \frac{d^2(x,y)}{T^{2H}}.
\]
If $y$ is not close to $x$, we just have to pick a sequence $x_0=x,\cdots,x_m=y$ such that
\[
d(x_i,x_{i+1}) \le \varepsilon 
\]
and 
\[
d(x,y)=\sum_{i=0}^{m-1}  d(x_i,x_{i+1}),
\]
where $\varepsilon$ is small enough. 
\end{proof}

The second keypoint is the following

\begin{theorem}\label{th: distance}
Fix $x_0\in M$. Let $F$ be a $C^\infty$ function on $M$. There exists a neighborhood $V$ of $x_0$ such that if $y_0\in V$ is a non-degenerate minimum of
$$F(y)+\frac{d^2(x_0,y)}{2},$$
then there exists a unique $k_0\in\msh_H$ such that
(a): $\Phi_1(x_0,k_0)=y_0$;
(b): $d(x_0,y_0)=\|k_0\|_{\msh_H}$; and (c): $k_0$ is a non-degenerate minimum of the functional: $k\rightarrow F(\Phi_1(x_0,k))+1/2\|k\|^2_{\msh_H}$ on $\msh_H$.
\end{theorem}
\begin{proof}
The first two statements are clear from Proposition \ref{th: ministance}. We only need to prove (c). To simplify notation, let
$$G(k)=F(\Phi_1(x_0,k))+\frac{1}{2}\|k\|^2_{\msh_H}.$$
Consider
$$u(t)=G(k_0+tk),$$
and
$$v(t)=F(\Phi_1(x_0,k_0+tk))+\frac{1}{2}d^2(x_0, \Phi_1(x_0,k_0+tk)).$$
It is clear that
$$u(t)\geq v(t), \quad u(0)=v(0)\quad\mathrm{and}\quad u'(0)=v'(0)=0.$$
Thus
$$d^2 G(k_0) k^2=u''(0)\geq v''(0)=\left(F+\frac{1}{2}d(x_0,\cdot)^2\right)''(y_0)\left(d\Phi_1(k_0) k\right)^2.$$
When $k\notin \mathrm{Ker}(d\Phi_1(x_0,k_0))$, we surely have 
$$d^2G(k_0)k^2>0.$$
In the case $k\in \mathrm{Ker}(d\Phi_1(x_0,k_0))$, we have 
\begin{align}\label{y0=x0}d^2G(k_0)k^2>0, \quad\quad\mathrm{when}\ y_0=x_0.\end{align}
To see this, first note that since $k\in \mathrm{Ker}(d\Phi_1(k_0,x_0))$ we can chose a family of path $\{z^t\in C([0,1];\mr^d); t\in[0,1]\}$ such that $z^t_0=z^t_1=z^0_s=0$ for all $(t,s)\in [0,1]\times[0,1]$, and
$$\frac{dz^t}{dt}\bigg|_{t=0}=d\Phi(x_0,k_0)k.$$
Moreover, we have $z^t=\Phi(0,k^t)$ for a family of path $k^t\in \msh_H$. Therefore
\begin{align*}
d^2G(k_0)k^2=\frac{d^2}{dt^2}\bigg|_{t=0}\left(F(x_0+z^t_1)+\frac{1}{2}\|k^t\|^2_{\msh_H}\right)
=\int_0^1\dot{\left[K^{-1}_H\left(\frac{d}{dt}\bigg|_{t=0}k^t\right)\right]}^2_sds.
\end{align*}
This shows that if $d^2G(k_0)k^2=0$ then $k=0$, which proves (\ref{y0=x0}). Now the lemma follows by a continuity argument. 
\end{proof}

\begin{remark}\label{remark V}
In the above lemma, it is clear that we can choose the neighborhood $V$ of $x_0$ such that for any $x\in V$, if $y\in V$ is a non-degenerate minimum of $F(y)+d(x,y)^2/2$, then the three properties in the lemma are fulfilled.
\end{remark}

\subsection{Asymptotics of the density function}
Consider 
$$dX^\varepsilon_t=\varepsilon\sum_{i=1}^{d}V_i(X^\varepsilon_t)dB^i_t\quad\quad\mathrm{with}\ X^\varepsilon_0=x.$$
Before applying the Laplace approximation to $X^\varepsilon_t$, we need the following lemma which gives us the correct functionals $F$ and $f$.
\begin{lemma}
Let $V$ be in Remark \ref{remark V}. There exists a bounded smooth function $F(x,y,z)$ on $V\times V\times M$ such that:

(1) For any $(x,y)\in V\times V$ the infimum 
$$\inf\left\{F(x,y,z)+\frac{d(x,z)^2}{2}, z\in M\right\}=0$$
is attained at the unique point $y$. More over, it is a non-degenerate minimum.

(2) For each $(x,y)\in V\times V$, there exists a ball centered at $y$ with radius $r$ independent of $x,y$ such that $F(x,y,\cdot)$ is a constant outside of the ball.
\end{lemma}
\begin{proof} 
See Lemma 3.8 in Ben Arous\cite{benarous2}.
\end{proof}

Let $F$ be in the above lemma and $p_\varepsilon(x,y)$ the density function of $X^\varepsilon_1$.  By the inversion of Fourier transformation we have
\begin{align*}
p_\varepsilon(x,y)e^{-\frac{F(x,y,y)}{\varepsilon^2}}&=\frac{1}{(2\pi)^d}\int e^{-i\zeta\cdot y}d\zeta\int e^{i\zeta\cdot z}e^{-\frac{F(x,y,z)}{\varepsilon^2}}p_\varepsilon (x,z)dz\\
&=\frac{1}{(2\pi\varepsilon)^d}\int e^{-i\frac{\zeta\cdot y}{\varepsilon}}d\zeta\int e^{i\frac{\zeta\cdot z}{\varepsilon}}e^{-\frac{F(x,y,z)}{\varepsilon^2}}p_\varepsilon(x,z)dz\\
&=\frac{1}{(2\pi\varepsilon)^d}\int d\zeta \me_x\left(e^{\frac{i\zeta\cdot(X_1^\varepsilon-y)}{\varepsilon}}e^{\frac{F(x,y,X^\varepsilon_1)}{\varepsilon^2}}\right).
\end{align*}
It is clear that by  applying Laplace approximation to the expectation in the last equation above and  
switching the order of integration (with respect to $\zeta$) and summation, we obtain an asymptotic expansion for the the density function $p_\varepsilon(x,y)$.  On the other hand, we cannot apply the Laplace method here directly since we need a uniform control in $x$ and $y$. Also we need to show that the use of Fourier inversion is legitimate.

To make the above prior computation rigorous, we modify the Laplace method in the previous section as follows. 

First note that by Lemma \ref{th: distance}, Assumption \ref{assumption Laplace} is satisfied.
Consider 
$$dZ_t^\varepsilon(x,y)=\sum_{i=1}^d V_i\big(Z^\varepsilon_t(x,y)\big)\big(\varepsilon dB^i_t+d\gamma_t^i(x,y)\big),\quad\quad\mathrm{with}\ Z_0^\varepsilon(x,y)=x.$$
In the above $(x,y)\in V\times V$ and $\gamma(x,y)$ is the unique path in $\msh_H$ such that $\Phi_1(x,\gamma(x,y))=y$ and $\|\gamma(x,y)\|_{\msh_H}=d(x,y).$ 
\begin{lemma}\label{th: Derivatives Z}
Let $Z^\varepsilon_t(x,y)$ be the process defined above, then $Z^\varepsilon_t(x,y)$ is $C^\infty$ in $(\varepsilon, x,y)$. Moreover, there exists an $\varepsilon_0>0$ such that
$$\sup_{\varepsilon\leq\varepsilon_0}\sup_{x,y\in V\times V}\sum_{j=0}^n \me\left(\sup_{t\in[0,1]}\|D^j(\partial_x^\alpha\partial_y^\beta\partial_\varepsilon^m Z_t^\varepsilon(x,y))\|_{\mathrm{HS}}^p\right)<\infty.$$
Here $m,n$ are non-negative integers, $p\geq 2$  and $\alpha\in\{1,2,...,d\}^k, \beta\in\{1,2,...,d\}^l$ are multiple indices. 
\end{lemma}
\begin{proof}
The first statement is clear. The second statement is a consequence of Lemma \ref{lem:apriori}
\end{proof}

Now consider the stochastic Taylor expansion for $Z^\varepsilon$
\begin{align}\label{Talyor Z}
Z_t^\varepsilon=\phi_t(x,y)+\sum_{j=1}^N\frac{g_t^k(x,y)\varepsilon^k}{k!}+R_t^{N+1}(\varepsilon,x,y)\varepsilon^{N+1}.
\end{align}
Here
$$\phi(x,y)=\Phi(x,\gamma(x,y)),$$
and
$$R_t^{N+1}(\varepsilon,x,y)=\int_0^1\partial_\varepsilon^{N+1}Z_t^\varepsilon(x,y)\frac{(1-v)^N}{N!}dv.$$

Let $$\theta(\varepsilon, x,y)=F(x,y,Z^\varepsilon_1(x,y)).$$
We have
$$\theta(\varepsilon,x,y)=\theta(0,x,y)+\varepsilon\partial_\varepsilon \theta(0,x,y)+\varepsilon^2U(\varepsilon,x,y).$$
where
$$U(\varepsilon,x,y)=\int_0^1\partial_\varepsilon^2\theta(\varepsilon,x,y)(1-v)dv.$$
By our choice of $Z^\varepsilon$, it is clear
\begin{align}\label{theta 0}
\theta(0,x,y)=F(x,y,\phi_1(x,y))=F(x,y,y).
\end{align}
Lemma \ref{th: theta 1} gives us
\begin{align}\label{theta 1}
\partial_\varepsilon\theta(0,x,y)=-\int_0^1 (K^*_H)^{-1}\dot{(K_H^{-1}\gamma(x,y))}_sdB_s.
\end{align}

Thus applying Cameron-Martin theorem for fBm (Theorem \ref{C-M}), we have
\begin{align*}&\me_x\exp\left( \frac{i\zeta\cdot(X_1^\varepsilon-y)}{\varepsilon}-\frac{F(x,y,X_1^\varepsilon)}{\varepsilon^2} \right)\\=&\me\left[\exp\left(\frac{i\zeta\cdot(Z_1^\varepsilon-y)}{\varepsilon}-\frac{F(x,y,Z_1^\varepsilon)}{\varepsilon^2} \right)\exp\left(-\frac{1}{\varepsilon}\int_0^1\big((K^*_H)^{-1}(\dot{K_H^{-1}{\gamma}})\big)_sdB_s-\frac{\|\gamma\|^2_{\msh_H}}{2\varepsilon^2}\right)\right]\\
=&\exp\left[-\frac{a}{\varepsilon^2}\right]\me_x\left[\exp\bigg(i\zeta\cdot g^1_1(x,y)\bigg)\exp\bigg(i\zeta\cdot V(\varepsilon,x,y)-U(\varepsilon,x,y)\bigg)\right].\end{align*}
In the above
$$a(x,y)=F(x,y,y)+\frac{d^2(x,y)}{2}=0,$$
and $$V(\varepsilon,x,y)=\frac{Z^\varepsilon_1(x,y)-y-\varepsilon g^1_1(x,y)}{\varepsilon}=\varepsilon R^2_1(\varepsilon,x,y).$$

Similar to the argument in Section 2, we need to estimate
$$\me_x\left[\exp\bigg(i\zeta\cdot g^1_1(x,y)\bigg)\exp\bigg(i\zeta\cdot V(\varepsilon,x,y)-U(\varepsilon,x,y)\bigg)\right].$$ 
For this purpose, we need 

\begin{lemma}\label{th: U(x,y)}
There exist $C>0$ and $\varepsilon_0>0$ such that
$$\sup_{(x,y)\in V\times V}\sup_{\varepsilon<\varepsilon_0}\me e^{-(1+C)U(\varepsilon,x,y)} <\infty.$$
\end{lemma}
\begin{proof}
We only sketch the proof. Details can be found in Ben Arous\cite{benarous2} (with minor modifications) and will not be repeated here.

Fix any $1/2<\lambda<H$. One can show that for $\rho>0$ there exist constants $C>0$, $b>0$ and $\varepsilon_0>0$ such that for all $\varepsilon<\varepsilon_0$ and all $(x,y)\in V\times V$ we have
\begin{align}\label{th: U1}
\me_x\left\{e^{-(1+C)U(\varepsilon,x,y)}; \left\|Z^\varepsilon_t-\phi_t(x,y)\right\|_{\lambda,1}\geq\rho\right\}\leq e^{\frac{-b}{\varepsilon^2}}.
\end{align}
Here $\|\cdot\|_{\lambda,t}$ is the $\lambda$-H\"{o}lder norm up to time $t$. The above estimate is a consequence of the following application of the large deviation principle to $X^\varepsilon_1$, i.e.,
$$\limsup_{\varepsilon\rightarrow 0}\varepsilon^2\log\me_x\left\{e^{-\frac{F(x,y,X^\varepsilon_1)}{\varepsilon^2}}; \|X^\varepsilon-\phi(x,y)\|_{\lambda,1}\geq \rho\right\}<-a(x,y)=0.$$

On the other hand, applying Lemma \ref{th: key lemma2} we have, for each $(x,y)\in V\times V$ there exists $C>0$ and $\varepsilon_0>0$ such that
$$\sup_{\varepsilon<\varepsilon_0}\me_x\left\{e^{-(1+C)U(\varepsilon,x,y)}; \|Z^\varepsilon-\phi(x,y)\|_{\lambda,1}\leq \rho\right\}<\infty.$$
Since we have smoothness of $Z^\varepsilon(x,y)$ (in $x$ and $y$) and $V\times V$ is contained in a compact subset of $M\times M$, the above estimate leads to
\begin{align*}
\sup_{\varepsilon<\varepsilon_0}\sup_{(x,y)\in V\times V}\me_x\left\{e^{-(1+C)U(\varepsilon,x,y)}; \|Z^\varepsilon-\phi(x,y)\|_{\lambda,1}\leq \rho\right\}<\infty.
\end{align*} 
Together with (\ref{th: U1}) the proof is completed. 

\end{proof}

Set
$$\Upsilon(\varepsilon,x,y)=e^{i\zeta\cdot V(\varepsilon,x,y)-U(\varepsilon,x,y)}$$
and consider the stochastic Taylor expansion for it
\begin{align}\label{upsilon}\Upsilon(\varepsilon,x,y,\zeta)=\sum_{m=0}^N\partial_\varepsilon^m\Upsilon(0,x,y,\zeta)\frac{\varepsilon^m}{m!}+S_{N+1}(\varepsilon,x,y,\zeta)\varepsilon^{N+1},\end{align}
where
$$S_{N+1}(\varepsilon,x,y,\zeta)=\int_0^1\partial_\varepsilon^{N+1}\Upsilon(\varepsilon v,x,y,\zeta)\frac{(1-v)^N}{N!}dv.$$

\begin{lemma}\label{th: Derivatives upsilon}
For any non-negative integers $k,l,m$ and $n$, and multi-indices $\alpha\in\{1,2,...,d\}^k$ and $\beta\in\{1,2,...,d\}^l$, we have

(1) For all $p\geq 2$, there exists $\varepsilon_0>0$ such that
$$\sup_{\varepsilon\leq\varepsilon_0}\sup_{x,y\in V\times V}\me\bigg(\sum_{j=0}^n\sup_{t\in[0,1]}\|D^j(\partial_x^\alpha\partial_y^\beta\partial_\varepsilon^m i\zeta\cdot V(\varepsilon,x,y)-U(\varepsilon,x,y)\|_{\mathrm{HS}}^p\bigg)<\infty.$$

(2) There exist $C>0, K>0$ and $\varepsilon_0>0$ such that
$$\sup_{\varepsilon\leq\varepsilon_0}\sup_{x,y\in V\times V}\me\bigg(\sum_{j=0}^n\sup_{t\in[0,1]}\|D^j(\partial_x^\alpha\partial_y^\beta\partial_\varepsilon^m \Upsilon(\varepsilon,x,y,\zeta)\|_{\mathrm{HS}}^{1+C}\bigg)<K\big(\|\zeta\|+1\big)^{m+k+l}.$$
Moreover, we have
$$\sup_{\varepsilon\leq\varepsilon_0}\sup_{x,y\in V\times V}\me\bigg(\sum_{j=0}^n\sup_{t\in[0,1]}\left\|D^j(\partial_x^\alpha\partial_y^\beta\partial_\varepsilon^m \big(e^{i\zeta\cdot g^1_1(x,y)}\Upsilon(\varepsilon,x,y,\zeta)\big)\right\|_{\mathrm{HS}}^{1+C}\bigg)<K\big(\|\zeta\|+1\big)^{m+k+l}.$$
\end{lemma}
\begin{proof}
We follow the argument in Ben Arous\cite{benarous2}. Note that
$$i\zeta\cdot V(\varepsilon,x,y)-U(\varepsilon,x,y)=i\zeta\int_0^1\partial_\varepsilon^2Z^{\varepsilon v}_1(x,y)(1-v)dv-\int_0^1\partial^2_\varepsilon\theta(\varepsilon v,x,y)(1-v)dv.$$
The estimate in (1) follows directly from Lemma \ref{th: Derivatives Z}.

For the second statement, first note that $$e^{-U}\in \mathrm{Dom}(D).$$
This is seen by an approximating argument and that $D$ is a closed operator. Moreover, we have
$$D (e^{-U})=-(DU)e^{-U}.$$
Hence $\Upsilon$ is also in the domain of $D$. 

It is clear that $\partial^\alpha_x\partial^\beta_y\partial^m_\varepsilon\Upsilon$ is of the form $W\Upsilon$, where $W$ is a polynomial in $\zeta$ of degree $m+|\alpha|+|\beta|$ with coefficients derivatives (w.r.t. $x, y$ and $\varepsilon$) of $U(\varepsilon,x,y)$ and $V(\varepsilon,x,y)$. Moreover,
$$D(\partial^\alpha_x\partial^\beta_y\partial^m_\varepsilon\Upsilon)=(DW+i\zeta\cdot DV-DU)\Upsilon.$$
The first estimate in (2) now follows immediately from (1) and Lemma \ref{th: U(x,y)}. The last estimate in (2) then follows from the first one in (2) and Lemma \ref{th: Derivatives Z}. This completes the proof.

\end{proof}
With the above lemma, we are now able to obtain an asymptotic expansion for 
$$\me_x\left[\exp\bigg(i\zeta\cdot g^1_1(x,y)\bigg)\exp\bigg(i\zeta\cdot V(\varepsilon,x,y)-U(\varepsilon,x,y)\bigg)\right].$$ 
Define
$$\alpha_m(x,y,\zeta)=\me_x\bigg[\exp\big(i\zeta\cdot g^1_1(x,y)\big)\partial^m_\varepsilon\Upsilon (0,x,y,\zeta)\bigg],$$
and
$$T_{N+1}(\varepsilon,x,y,\zeta)=\me_x\bigg[\exp\big(i\zeta\cdot g^1_1(x,y)\big)S_{N+1}(\varepsilon,x,y,\zeta)\bigg].$$
Recall (\ref{upsilon}), we obtain
\begin{align*}
&\me_x\left[\exp\bigg(i\zeta\cdot g^1_1(x,y)\bigg)\exp\bigg(i\zeta\cdot V(\varepsilon,x,y)-U(\varepsilon,x,y)\bigg)\right]\\
=&\me_x\left[\exp\bigg(i\zeta\cdot g^1_1(x,y)\bigg)\Upsilon(\varepsilon,x,y,\zeta)\right]\\
=&\sum_{m=0}^{N}\alpha_m(x,y,\zeta)\varepsilon^m+T_{N+1}(\varepsilon,x,y,\zeta)\varepsilon^{N+1}.
\end{align*}
\begin{remark}
Indeed, Lemma \ref{th: Derivatives upsilon} provides us smoothness and boundedness of $\alpha_m$ and $T_{N+1}$.
\end{remark}
So far, we have obtained that for all $\zeta\in\mr^d$
\begin{align*}
&\me_x\exp\left( \frac{i\zeta\cdot(X_1^\varepsilon-y)}{\varepsilon}-\frac{F(x,y,X_1^\varepsilon)}{\varepsilon^2} \right)\\
=&e^{-\frac{a(x,y)}{\varepsilon^2}}\bigg(\sum_{m=0}^{N}\alpha_m(x,y,\zeta)\varepsilon^m+T_{N+1}(\varepsilon,x,y,\zeta)\varepsilon^{N+1}\bigg)\\
=&\sum_{m=0}^{N}\alpha_m(x,y,\zeta)\varepsilon^m+T_{N+1}(\varepsilon,x,y,\zeta)\varepsilon^{N+1}.
\end{align*}
To apply the inversion of Fourier transformation, we need integrability of $\alpha_m$ and $T_{N+1}$ in $\zeta$, which is answered in the following lemma.

\begin{lemma}
For any non-negative integers $p, k$ and $l$, and multi-indices $\alpha\in\{1,2,...,d\}^k$ and $\beta\in\{1,2,...,d\}^l$, we have

(1) There exists $K=K_p(\alpha,\beta)>0$ such that
$$\sup_{(x,y)\in V\times V}\left|\partial_x^\alpha\partial_y^\beta\alpha_m(x,y,\zeta)\right|\leq\frac{K}{\|\zeta\|^{2p}}\left(\|\zeta\| +1\right)^{m+k+l}.$$

(2) There exists $\varepsilon_0>0$ and $K=K(p,N,\alpha,\beta,m)>0$ such that 
$$\sup_{\varepsilon<\varepsilon_0}\sup_{(x,y)\in V\times V}\left|\partial_x^\alpha\partial_y^\beta\partial_\varepsilon^m T_{N+1}(\varepsilon,x,y,\zeta)\right|\leq\frac{K}{\|\zeta\|^{2p}}\left(\|\zeta\| + 1\right)^{(N+1)+k+l}.$$

\end{lemma}
\begin{proof}
The lemma follows from integration by parts in Malliavin calculus. Indeed, first note that by equation (\ref{g1}), the Malliavin matrix of $g^1$ is deterministic, non-degenerate and uniform in $x$ and $y$. By Proposition 5.7 and Proposition 5.8 in Shigekawa\cite{Shigekawa} and Lemma \ref{th: Derivatives Z}, for any proper test function $\psi$, $G\in \md^{|\alpha|,q}$, there exist $l_\alpha G$ and $r<q$ such that
$$\me\big(\partial^\alpha\psi(g^1_1)G\big)=E\big(\psi(g^1_1)l_\alpha(G)\big)$$
and
$$\big(\me|l_\alpha(G)|^r\big)^{\frac{1}{r}}\leq K\left(\sum_{j=0}^{|\alpha|}\me\|D^jG\|^q_{\mathrm{HS}}\right)^{\frac{1}{q}}.$$
Here $K$ depends on$|\alpha|$, $g^1_1$ and its Malliavin matrix and $K$ is uniform in $x$ and $y$.

Applying the above integration by parts formula with $$\psi(u)=e^{i\zeta\cdot u}\quad\quad\mathrm{and}\quad \partial^\alpha=\left(\sum_{i=1}^d\partial^2_{u_i}\right)^p.$$
We have
$$\left|\me\big(e^{i\zeta\cdot g^1_1}G\big)\right|\leq \frac{K}{\|\zeta\|^{2p}}\left(\sum_{j=0}^{2p}\me\big(\|D^jG\|^q_{\mathrm{HS}}\big)\right)^{\frac{1}{q}}.$$
Now the lemma follows by Lemma \ref{th: Derivatives upsilon} and replacing $G$ in the above by
$$G_1=\partial^\alpha_x\partial^\beta_y\partial^m_\varepsilon\Upsilon(0,x,y,\zeta),$$
and
$$G_2=\partial^\alpha_x\partial^\beta_y\partial^m_\varepsilon\big(S_{N+1}(\varepsilon,x,y,\zeta)e^{i\zeta\cdot g^1_1}\big)e^{-i\zeta\cdot g^1_1}.$$

\end{proof}

Now we only need to chose $2p> d+(N+1)+k+l$ in the previous lemma and obtain
\begin{align*}
p_\varepsilon(x,y)e^{-\frac{F(x,y,y)}{\varepsilon^2}}=\frac{e^{-\frac{a(x,y)}{\varepsilon^2}}}{\varepsilon^d}\bigg(\sum_{m=0}^{N}\beta_m(x,y)\varepsilon^m+t_{N+1}(\varepsilon,x,y)\varepsilon^{N+1}\bigg).
\end{align*}
Here
$$\beta_m(x,y)=\frac{1}{(2\pi)^d}\int \alpha_m(x,y,\zeta)d\zeta,$$
and
$$t_{N+1}(\varepsilon,x,y)=\frac{1}{(2\pi)^d}\int T_{N+1}(\varepsilon,x,y,\zeta)d\zeta.$$
Notice that the $\beta_m(x,y,\zeta)$ is an odd function in $\zeta$ when $m$ is odd (cf, Ben Arous\cite{benarous2}). Now by the self-similarity of the fractional Brownian motion and it $\varepsilon= t^{H}$ we obtain the desired asymptotic formula for the density function.

\subsection{The on-diagonal asymptotics}

As a straightforward corollary of Theorem \ref{main}, we have the following on-diagonal asymptotics:
\[
p(t;x,x)=\frac{1}{t^{Hd}} \left( a_0(x)+a_1(x)t^{2H}+\cdots+a_n(x)t^{2nH}+o(t^{2nH}) \right).
\]
In this subsection, we analyze the coefficients $a_n(x)$ and show how they are related to some functionals of the underlying  fractional Brownian motion.

We first introduce some notations and remind some results that may be found in \cite{Bau}, \cite{Bau-Cou}, \cite{neuen} and \cite{Friz} 

If $I=(i_1,...,i_k) \in \{ 1,..., d \}^k$ is a word, we denote by
$V_I$ the Lie commutator defined by
\[
V_I = [V_{i_1},[V_{i_2},...,[V_{i_{k-1}}, V_{i_{k}}]...].
\]
The group of permutations of the set $\{ 1, ..., k \}$ is denoted
$\mathfrak{S}_k$. If $\sigma \in \mathfrak{S}_k$, we denote by
$e(\sigma)$ the cardinality of the set
\[
\{ j \in \{1,...,k-1 \} , \sigma (j) > \sigma(j+1) \}.
\]
Finally, for the iterated integrals, defined in Young's
sense, we use the following notations:
\begin{enumerate}
\item
\[
\Delta^k [0,t]=\{ (t_1,...,t_k) \in [0,t]^k, t_1 \leq ... \leq t_k
\};
\]
\item If $I=(i_1,...i_k) \in \{1,...,d\}^k$ is a word with length
$k$,
\begin{equation*}
\int_{\Delta^k [0,t]} dB^I= \int_{0 \leq t_1 \leq ... \leq t_k \leq t}
dB^{i_1}_{t_1} \cdots dB^{i_k}_{t_k}.
\end{equation*}
\item If $I=(i_1,...i_k) \in \{1,...,d\}^k$ is a word with length
$k$,
\[
\Lambda_I (B)_t= \sum_{\sigma \in \mathcal{S}_k} \frac{\left(
-1\right) ^{e(\sigma )}}{k^{2}\left(
\begin{array}{l}
k-1 \\
e(\sigma )
\end{array}
\right) } \int_{0 \leq t_1 \leq ... \leq t_k \leq t}
dB^{\sigma^{-1}(i_1)}_{t_1}  \cdots
dB^{\sigma^{-1}(i_k)}_{t_k},\quad t \ge 0.
\]
\end{enumerate}

\begin{theorem}
\label{thme-dl-pt} For  $f \in C^{\infty}_b( {\mathbb R}^d,
{\mathbb R})$ , $x \in {\mathbb R}^d$, and $N \geq 0$, when $t
\rightarrow 0$,
\begin{align*}
f(X^x_t) & =f(x) + \sum_{k=1}^N t^{2kH}
\sum_{I=(i_1,...i_{2k} ) } (V_{i_1} ... V_{i_{2k}} f) (x)
\int_{\Delta^{2k} [0,1]}  dB^I
+o(t^{(2N+1)H}) \\
&=f \left( \exp \left( \sum_{I, \mid I \mid  \le N }
\Lambda_I (B)_t V_{I} \right) x \right)+ o(t^{NH})
\end{align*}
and
\begin{align*}
\mathbb{E} (f(X^x_t))& =f(x) + \sum_{k=1}^N t^{2kH}
\sum_{I=(i_1,...i_{2k} ) } (V_{i_1} ... V_{i_{2k}} f) (x)
\mathbb{E}\left(\int_{\Delta^{2k} [0,1]}  dB^I \right)
+o(t^{(2N+1)H}) \\
& =\mathbb{E} \left( f \left( \exp \left( \sum_{I, \mid I \mid  \le N }
\Lambda_I (B)_t V_{I} \right) x \right) \right)+ o(t^{NH})
\end{align*}
\end{theorem}

As a consequence, we obtain the following proposition which may be proved as in \cite{cours} (or  \cite{Kunita}).
\begin{proposition}\label{diffusion tangente}
For $N \ge 1$, when $t \rightarrow 0$,
\[
p (t;x_0,x_0)=d^N_t (x_0)+ O\left( t^{H(N+1-d)}\right),
\]
where $d^N_t(x_0)$ is the density at $0$ of the random variable $
\sum_{I, \mid I \mid \le N } \Lambda_I (B)_t V_{I}(x_0)$
\end{proposition}

This proposition may be used to understand the geometric meaning of the coefficients $a_k(x_0)$ of the small-time asymptotics 
\[
p(t;x,x)=\frac{1}{t^{Hd}} \left( a_0(x)+a_1(x)t^{2H}+\cdots+a_n(x)t^{2nH}+o(t^{2nH}) \right).
\]
For instance, by applying the previous proposition with $N=1$, we get
\[
a_0(x_0)=\frac{1}{(2 \pi )^{\frac{d}{2}}} \frac{1}{| \det (V_1(x_0),\cdots,V_d(x_0))|}
\]
The computation of $a_1(x)$ is technically more involved.
We wish to apply the previous proposition with $N=2$. For that, we need to understand the law of the random variable
\[
\Theta_t=\sum_{i=1}^d B^i_t V_i (x_0)+\frac{1}{2} \sum_{1\le i<j \le d}\int_0^t B^i_s dB^j_s -B^j_sdB_i^s [V_i,V_j](x_0).
\]
From the structure equations, we  have
\[
\Theta_t=\sum_{k=1}^d \left( B^k_t +\frac{1}{2} \sum_{1\le i<j \le d}\omega_{ij}^k\int_0^t B^i_s dB^j_s -B^j_sdB^i_s \right) V_k (x_0).
\]
By a simple linear transformation, we are reduced to the problem of the computation of the law of the $\mathbb{R}^d$-valued random variable
\[
\theta_t=\left(B^k_t +\frac{1}{2} \sum_{1\le i<j \le d}\omega_{ij}^k\int_0^t B^i_s dB^j_s -B^j_sdB^i_s \right)_{1 \le k \le d}.
\]
At that time, up to the knowledge of the authors, there is no explicit formula for this distribution.  However, the scaling property of fractional Brownian motion and the inverse Fourier transform formula leads easily to the following expression
\[
p_t(x_0,x_0)=\frac{1}{| \det (V_1(x_0),\cdots,V_d(x_0))|}\frac{1}{(2 \pi t^{2H} )^{d/2}} \left( 1- q_H(\omega) t^{2H} +o(t^{2H}) \right),
\]
where $q_H(\omega)$ is the quadratic form given by
\[
q_H(\omega)=\frac{1}{8(2 \pi )^{\frac{d}{2}}}\int_{\mathbb{R}^d} \mathbb{E} \left( e^{i\langle \lambda, B_1 \rangle} \left( \sum_{1 \le i <j \le d} \langle \omega_{ij}, \lambda \rangle \int_0^1 B^i_s dB^j_s -B^j_sdB^i_s\right)^2\right)d \lambda.
\]

\section{Appendix}
In this last section, we give proofs for the technical lemmas we used before. 

Fix $1/2<\lambda<H$. Let $B(\phi, \rho)\in C^\lambda(0,T;\mr^d)$ be the ball centered at $\phi$ with radius $\rho$ under the $\lambda$-H\"{o}lder topology
$$\|f\|_{\lambda,T}:=\|f\|_\infty+\sup_{0\leq s<t\leq T}\frac{|f(t)-f(s)|}{(t-s)^\lambda},\quad \mathrm{for\ all}\ f\in C^\lambda(0, T; \mr^d).$$
Note that the $\lambda$-H\"{o}lder topology is a stronger topology than the usual supreme topology.

Recall the two expressions for $Z^\varepsilon$
\begin{align}\label{Z equ}
dZ^\varepsilon_t=\sigma(Z^\varepsilon_t)(\varepsilon dB_t+d\gamma_t)+b(\varepsilon, Z^\varepsilon_t)dt
\end{align}
and
\begin{align}\label{Z}Z^\varepsilon=\phi+\sum_{j=0}^N\frac{g_j\varepsilon^j}{j!}+\varepsilon^{N+1}R^\varepsilon_{N+1}.\end{align} 
Here $\gamma\in\msh_H$, hence $\gamma\in I_0^{H+1/2}(L^2)\subset C^H(0,T; \mr^d).$

\subsection{Proof of Lemma \ref{th: key lemma1}} 
We show, for all $t\in[0,T]$, there exists a constant $C$ such that for $r$ large enough we have
\begin{align}\label{key est1-1}
&\mp\{\|g_1\|_{\lambda,t}\geq r\}\leq\exp\left\{-\frac{C r^2}{t^{2H}}\right\}\\
&\mp\{\|g_2\|_{\lambda,t}\geq r\}\leq\exp\left\{-\frac{C r}{t^{2H}}\right\},\nonumber
\end{align}
and
\begin{align}\label{key est1-2}
&\mp\{\|\varepsilon R_1^\varepsilon\|_{\lambda, t}\geq r; t\leq T^\varepsilon\}\leq \rho\\
&\mp\{\|\varepsilon R_2^\varepsilon\|_{\lambda,t}\geq r; t\leq T^\varepsilon\}\leq \exp\left\{-\frac{C r^2}{\rho^2 t^{2H}}\right\}\nonumber \\
&\mp\{\|\varepsilon R_3^\varepsilon\|_{\lambda,t}\geq r; t\leq T^\varepsilon\}\leq \exp\left\{-\frac{C r}{\rho t^{2H}}\right\}.\nonumber
\end{align}
Here $T^\varepsilon$ is the first exist time of $Z^\varepsilon$ from  $B(\phi,\rho)$.

We first prove the estimates for $g_i$'s. Write
\begin{align}\label{sigma1}
\sigma(Z^\varepsilon)=\sigma(\phi)+\sigma_x(\phi)(Z^\varepsilon-\phi)+\frac{1}{2}\sigma_{xx}(\phi)(Z^\varepsilon-\phi)^2+O(\varepsilon^3)
\end{align}
and
\begin{align}\label{b1}
b(\varepsilon, z^\varepsilon)=b(0,\phi)&+b_x(0,\phi)(Z^\varepsilon-\phi)+\frac{1}{2}b_{xx}(0,\phi)(Z^\varepsilon-\phi)^2+O(\varepsilon^3)\\
&\quad\ \quad\quad +b_\varepsilon(0,\phi)\varepsilon+\ \ b_{\varepsilon x}(0,\phi)(Z^\varepsilon-\phi)\varepsilon+O(\varepsilon^3)\nonumber\\&\quad\quad\ \ \quad\quad\quad\quad \quad\quad\quad\ \ \ \ +\ \ \frac{1}{2}b_{\varepsilon\varepsilon}(0,\phi)\varepsilon^2 +O(\varepsilon^3)\nonumber.
\end{align}
Substituting into the two expressions of $Z^\varepsilon$ gives us
\begin{align}\label{g1}
dg_1(s)=\sigma(\phi_s)dB_s+\sigma_x(\phi_s)g_1(s)d\gamma_s+b_x(0,\phi_s)g_1(s)ds+b_\varepsilon(0,\phi_s)ds.
\end{align}
and
\begin{align}\label{g2}
dg_2(s)=&2\sigma_x(\phi_s)g_1(s)dB_s+\sigma_{xx}(\phi_s)g_1(s)^2d\gamma_s+\sigma_x(\phi_s)g_2(s)d\gamma_s\\
&+b_{xx}(0,\phi_s)g_1(s)^2ds+b_x(0,\phi_s)g_2(s)ds+b_{\varepsilon\varepsilon}(0,\phi_s)ds\nonumber\\
&+2b_{\varepsilon x}(0,\phi_s)g_1(s)ds.\nonumber
\end{align}
By (\ref{g1}) and Remark \ref{remark 1}, it is clear that 
$$\|g_1\|_{\lambda, t}\leq C\|B\|_{\lambda, t},\quad\quad t\in[0,T],$$
where $C$ is a constant depending only on $\|\phi\|_{\lambda,T}$, $\|\gamma\|_{\lambda, T}$ and $T$. This gives us the first estimate in (\ref{key est1-1}).

Similarly, by (\ref{g2}) and Remark \ref{remark 1} together with the estimate we just obtained for $g_1$, we have 
\begin{align*}\|g_2\|_{\lambda,t}&\leq C(1+\|g_1\|_{\lambda,t}+\|g_1\|^2_{\lambda,t}+\|g_1\|_{\lambda,t}\|B\|_{\lambda,t})\\
&\leq C\|B\|_{\lambda, t}^2.
\end{align*}
Here $C$ is also constant depending only on $\|\phi\|_{\lambda,T}$, $\|\gamma\|_{\lambda, T}$ and $T$. Hence we have proved (\ref{key est1-1}).

In what follows, we prove (\ref{key est1-2}). To lighten our notation, in discussion that follows, we suppress the supper-script $\varepsilon$ in $R_i^\varepsilon$ whenever there is no confusion. 

Since we work in $B(\phi, \rho)$, the first inequality in (\ref{key est1-2}) is apparent.  We therefore only need to concentrate on the last two inequalities. 

First we use similar idea to deduce the equations satisfied by $R_i, i=1,2,3.$ For this purpose, define $\mu_1, \mu_2$ and $\nu_1, \nu_2$ by
\begin{align}\label{sigma2}
\sigma(Z^\varepsilon)&=\sigma(\phi)+\mu_1\varepsilon\\
&=\sigma(\phi)+\sigma_x(\phi)(Z^\varepsilon-\phi)+\mu_2\varepsilon^2.\nonumber
\end{align}
and
\begin{align}\label{b2}
b(\varepsilon, Z^\varepsilon)&=b(0,\phi)+\nu_1\varepsilon\\
&=b(0,\phi)+b_x(0,\phi)(Z^\varepsilon-\phi)+b_\varepsilon(0,\phi)\varepsilon+\nu_2\varepsilon^2.\nonumber
\end{align}
It is clear that $\mu_i ; i=1,2$ (resp. $\nu_i$) are of the form $\psi_i^\mu(\varepsilon R_1)(R_1)^i$ (resp. $\psi_i^\nu(\varepsilon R_1)(R_1)^i$), where $\psi_i$ are some functions of bounded derivatives determined by $\sigma$ and $b$. Hence in $B(\phi,\rho)$, $\mu_1, \nu_1$ are functions of $R_1$ with bounded derivatives, and there exists a constant $C$, depending only on derivatives of $\sigma$ and $b$, such that
\begin{align}\label{mu,nu}
\|\mu_1\|_{\lambda,t}, \|\nu_1\|_{\lambda,t}\leq C(1+\|R_1\|_{\lambda,t})\quad \mathrm{and}\quad \|\mu_2\|_{\lambda,t}, \|\nu_2\|_{\lambda,t}\leq C(1+\|R_1\|_{\lambda,t})^2.
\end{align}

Equations (\ref{Z}), (\ref{Z equ}), (\ref{sigma2}) and (\ref{b2}) give us
\begin{align}\label{R_1 R_2}
&dR_1(s)=\sigma(Z^\varepsilon_s)dB_s+\mu_1d\gamma_s+\nu_1ds\\
&dR_2(s)=2\mu_1dB_s+2\mu_2d\gamma_s+\sigma_x(\phi)R_2d\gamma_t+b_x(0,\phi_s)R_2ds+2\nu_2ds.\nonumber
\end{align}
Since we work with in $B(\phi, \rho)$, we have
$$\|Z^\varepsilon\|_{\lambda,t}\leq \|\phi\|_{\lambda,t}+\rho$$
hence
\begin{align*}\left\|\int_0^t\sigma(Z^\varepsilon_s)dB_s\right\|_{\lambda,t}<C\|B\|_{\lambda,t}\end{align*}
for some constant $C$ depending only on $\rho$, $\phi$ and derivatives of $\sigma$.
By standard Picard's iteration, we conclude that
\begin{align}\label{R1}\|R_1\|_ {\lambda,t}< C\|B\|_{\lambda,t}(1+\|\gamma\|_{\lambda,t}),\quad\quad\mathrm{in}\ B(\phi,\rho)\end{align}
for some constant $C$ uniformly bounded in $\varepsilon$.

The equation for $R_2$ is
\begin{align}\label{R2 equ}dR_2(s)-\sigma_x(\phi)R_2d\gamma_s-b_x(0,\phi_s)R_2ds=2\mu_1dB_s+2\mu_2d\gamma_s+2\nu_2ds.\end{align}
Recall that $\mu_1$ is of the form $\psi_1(\varepsilon R_1)R_1$, and in $B(\phi,\rho)$,
$$\|\varepsilon R_1\|_{\lambda,t}<\rho.$$
We obtain
$$\left\|\varepsilon\int_0^t\mu_1dB_s\right\|_{\lambda,t}=\left\|\int_0^t\psi_1^\mu(\varepsilon R_1)(\varepsilon R_1)dB_s\right\|_{\lambda,t}<\rho^2 C\|B\|_{\lambda,t}$$
for some constant $C$ uniformly bounded in $\varepsilon$. Similarly,
\begin{align*}\left\|\varepsilon\int_0^t\mu_2d\gamma_s+\varepsilon\int_0^t\nu_2ds\right\|_\alpha&=\left\|\int_0^t\psi_2^\mu(\varepsilon R_1)(\varepsilon R_1^2)d\gamma_s+\int_0^t\psi_2^\nu(\varepsilon R_1)(\varepsilon R_1^2)ds\right\|_{\lambda,t}\\
&\leq \rho^2 C\|R_1\|_{\lambda,t}(1+\|\gamma\|_{\lambda,t})\\
&\leq \rho^2 C\|B\|_{\lambda,t}(1+\|\gamma\|_{\lambda,t})
\end{align*}
for some constant $C$ uniformly bounded in $\varepsilon$. Hence by multiplying a factor $$\exp{\left\{-\int\sigma_x(\phi)d\gamma-\int b_x(0,\phi)du\right\}}$$ to both sides of (\ref{R2 equ}) and integrating from $0$ to $t$, we conclude
\begin{align*}
\mp\{\|\varepsilon R_2\|_{\lambda, t}\geq r; t\leq T^\varepsilon\}\leq \exp\left\{-\frac{C r^2}{\rho^2 t^{2H}}\right\}.
\end{align*}
This gives us the desired estimate for $\varepsilon R_2$. A similar argument also gives us
\begin{align*}
\mp\{\|R_2\|_{\lambda,t}\geq r; t\leq T^\varepsilon\}\leq \exp\left\{-\frac{C r}{\rho t^{2H}}\right\}.
\end{align*}

Continuing this type of argument, the equation for $R_3$ is given by
\begin{align*}
dR_3(s)-\sigma_x(\phi)R_3d\gamma_s-b_\varepsilon(0,\phi)R_3ds&=\sigma(\phi)R_2dB_s+\mu_2dB_s+\mu_3d\gamma_s+\frac{1}{2}\sigma_{xx}(\phi)R_1R_2d\gamma_s\\
&\ \ \quad\quad +\frac{1}{4}b_{xx}(0,\phi)R_1R_2ds+\nu_3ds+b_{\varepsilon,x}(0,\phi)R_2ds.
\end{align*}
By (\ref{mu,nu}) and (\ref{R1}) we conclude that in $B(\phi,\rho)$ we have for all $0<\varepsilon\le\rho$
$$\left\|\varepsilon\int_0^t\sigma(\phi)R_2+\mu_2dB_s\right\|_{\lambda,t}< \rho C\|B\|^2_{\lambda,t},$$
and similar estimates for the rest of the terms on the right hand side of the equation for $R_3$. Hence
\begin{align*}
\mp\{\|\varepsilon R_3\|_{\lambda,t}\geq r; t\leq T^\varepsilon\}\leq \exp\left\{-\frac{C r}{\rho t^{2H}}\right\}.
\end{align*} 
Therefore, we have proved (\ref{key est1-2}).

\subsection{Proof of Lemma \ref{th: key lemma2}}
For the convenience of quick reference, we re-state the lemma here.
\begin{lemma}\label{th: key est2}
Let$$\theta(\varepsilon)=F(Z^\varepsilon)=\theta(0)+\varepsilon\theta'(0)+\varepsilon^2U(\varepsilon)$$
where $U(\varepsilon)=\int_0^1(1-v)\theta''(\varepsilon v)dv$. There exist $\beta>0$ and $\varepsilon_0>0$ such that
$$\sup_{0\leq \varepsilon\leq\varepsilon_0}\me\left(e^{-(1+\beta)U(\varepsilon)}; Z^\varepsilon\in B(\phi,\rho)\right)<\infty.$$
\end{lemma}
Observe that 
\begin{align*}
(Z^\varepsilon-\phi)^2=\varepsilon^2g_1^2+\frac{1}{2}\varepsilon^3g_1R_2+\frac{1}{2}\varepsilon^3R_1R_2.
\end{align*} 
Thus, if we write
$$U(\varepsilon)=\frac{1}{2}\theta''(0)+\varepsilon R(\varepsilon)$$
then

\begin{align}\label{R}|R(\varepsilon)|\leq C(|R_3|+|g_1||R_2|+|R_2||R_1|+|R_1^3|).\end{align}
Together with the fact that $|\varepsilon R_1|\leq\rho$ we obtain
$$|\varepsilon R(\varepsilon)|\leq C(|\varepsilon R_3|+|g_1||\varepsilon R_2|+\rho|R_2|+\rho|R_1^2|).$$
Hence, from the estimates in Lemma \ref{th: key lemma1}, we conclude that for each $\alpha>0$, there exists  $\rho(\alpha)$ such that for all $\varepsilon\leq\rho\leq\rho(\alpha)$, we have
$$\sup_{0\leq \varepsilon\leq\rho}\me\left(e^{(1+\alpha)|\varepsilon R(\varepsilon)|}; t\leq T^\varepsilon\right)<\infty.$$ Therefore, to prove Lemma \ref{th: key lemma2} is reduced to prove the following
\begin{lemma}
There exists a $\beta>0$ such that
$$\me\exp\left\{-(1+\beta)\left[\frac{1}{2}\theta''(0)\right]\right\}<\infty.$$
\end{lemma}

\begin{proof}
We follow the proof in Ben Arous\cite{Ben Arous}. Since
$$U(0)=\frac{1}{2}\theta''(0)=\frac{1}{2}\bigg[dF(\theta)g_2+d^2F(\theta)g^2_1\bigg],$$
it is clear that to prove the above lemma, it suffices to prove that for sufficiently large $r$ we have
\begin{align}\label{ineq 1}\mp\left\{-\frac{1}{2}\bigg[dF(\phi)g_2+d^2F(\phi)g^2_1\bigg]\geq r\right\}\leq e^{-C r}, \quad\mathrm{with}\ C>1.\end{align}

Set $$Y^\varepsilon=(\varepsilon g_1, \varepsilon^2 g_2)$$
with $$dY^\varepsilon_s=\varepsilon\bar{\sigma}(s, Y^\varepsilon)dB_s+\bar{b}(\varepsilon, s, Y^\varepsilon)ds,\quad\quad Y^\varepsilon_0=0.$$
Here $\bar{\sigma}$ and $\bar{b}$ are determined by (\ref{g1}) and (\ref{g2}).   Define $A\subset C([0,T], \mr^{2d})$ by
$$A=\{\psi=(\psi_1,\psi_2)\in C([0,T], \mr^d\times\mr^d):\ dF(\phi)\psi_2+d^2F(\phi)\psi^2_1\leq -2\}.$$ 
We have
$$\mp\{Y^\varepsilon\in A\}=\mp\left\{-\frac{1}{2}\bigg[dF(\phi)g_2+d^2F(\phi)g^2_1\bigg]\geq\frac{1}{\varepsilon^2}\right\}$$
and by the large deviation principle for $Y^\varepsilon$
$$\limsup_{\varepsilon\rightarrow 0}\varepsilon^2\log\mp\{Y^\varepsilon\in A\}\leq -\Lambda^*(A).$$
Here $\Lambda^*$ is the good rate function of $Y^\varepsilon$. It is clear that to prove inequality (\ref{ineq 1}) it suffices to prove that $\Lambda^*(A)>1$.

Recall that $$\Lambda^*(A)=\inf\left\{\frac{1}{2}|k|_{\msh_H}^2; \Phi^*(k)\in A\right\}$$
where $u=\Phi^*(k)$ is the solution to the ordinary differential euqaiton
$$du_s=\bar{\sigma}(s,u_s)dk_s+\bar{b}(0,s,u_s)ds,\quad\mathrm{with}\ u_0=0.$$
It is easy to see from (\ref{chi}), (\ref{psi}), (\ref{g1}) and (\ref{g2}) that we have explicitly
$$u=(d\Phi(\gamma)k, d^2\Phi(\gamma)k^2).$$
By our assumption H 2 and the explanation after it, there exists $\nu\in (0,1)$ such that for all $k\in\msh_H-\{0\}$ we have
$$d^2F\circ\Phi(\gamma)k^2> (-1+\nu)|k|_{\msh_H}^2,$$
or
$$|k|^2_{\msh_H}>-\frac{1}{1-\nu}(d^2F\circ\Phi(\gamma)k^2)=-\frac{1}{1-\nu}(d^2F(\phi)(d\Phi(\gamma)k)+dF(\phi)(d^2\Phi(\gamma)k^2)).$$
Therefore, if $\Phi^*(k)\in A$, we have 
$$\frac{1}{2}|k|_{\msh_H}>\frac{1}{1-\nu}>1,$$
which implies $\Lambda^*(A)>1$ and completes the proof.
\end{proof}

\end{document}